\documentstyle[amsfonts, amssymb, latexsym, 12pt]{amsart}

\newtheorem{theorem}{Theorem}[section]
\newtheorem{lemma}[theorem]{Lemma}
\newtheorem{proposition}[theorem]{Proposition}

\newtheorem{definition}[theorem]{Definition}
\def\C{{\mbox{\rm\kern.24em
\vrule width.03em height1.43ex depth-.052ex \kern-.26em C}}}
\def\QSet{\mbox{\rm\kern.24em
\vrule width.03em height1.48ex depth-.051ex \kern-.26em Q}}
\def\Z{{\mbox{\rm\kern.25em
\vrule width.03em height0.57ex depth0ex
\kern.033em
\vrule width.03em height1.52ex depth-0.96ex \kern-.338em Z}}}

\def\R{{\mbox{\rm I\kern-.22em R}}}
\def\N{{\mbox{\rm I\kern-.22em N}}}
\def\P{{\bf P}}
\def\Q{{\bf Q}}
\def\\D{{\bf D}}
\def\T{{\bf T}}
\def\P{{\bf P}}
\def\supp{{\rm supp}}
\def\size{{\rm size}}

\def\energy{{\rm energy}}

\def\n{{\bf n}}
\def\s{{\bf s}}

\def\A{{\cal{A}}}
\def\D{{\cal{D}}}

\def\I{{\mathbb{I}}}
\def\C{{\cal{C}}}

\def\l{{\ell}}

\def\dist{{\rm dist}}

\def\111{\gamma}

\def\be{\begin{equation}\label}
\def\ee{\end{equation}}
\def\ban{\begin{align}}
\def\ean{\end{align}}
\def\beq{\begin{eqnarray*}}
\def\eeq{\end{eqnarray*}}
\def\bi{\begin{itemize}}
\def\ei{\end{itemize}}
\newenvironment{proof}{\noindent {\bf Proof.} }{\endprf\par}
\def \endprf{\hfill  {\vrule height6pt width6pt depth0pt}\medskip}
\def\emph#1{{\it #1}}

\title[Iterated trilinear Fourier integrals with arbitrary symbols]{\;\;\;\;\;\; Iterated trilinear Fourier integrals with\;\;\;\;\; arbitrary symbols}

\author{Joeun Jung}
\address{Department of Mathematics, Cornell University, Ithaca, NY 14853}
\email{joeunj.at.math.cornell.edu}

\begin{document}
\maketitle{}


\begin{abstract}

We prove 
$L^p$ estimates for trilinear multiplier operators with singular symbols. These operators arise in the study of iterated trilinear Fourier integrals, which are trilinear variants of the bilinear Hilbert transform. Specifically, we consider trilinear operators determined by multipliers that are products of two functions ${{m}}_1(\xi_1, \xi_2)$ and ${m}_2(\xi_2, \xi_3)$, such that the singular set of $m_1$ lies in the hyperplane $ \xi_1=\xi_2$ and that of $m_2$ lies in the hyperplane $\xi_2=\xi_3$. While previous work \cite{MTT2} requires that the multipliers satisfy $\chi_{\xi_1 <\xi_2} \cdot \chi_{\xi_2<\xi_3}$, our results allow for the case of the arbitrary multipliers, which have common singularities.
\end{abstract}


\section{Background}

$L^p$ estimates for multilinear singular operators and their connections to other fields such as partial differential equations, ergodic theory, and probability, have been the focus of a great deal of mathematical activity in recent years.

In 1979, Coifman and Meyer \cite{CM} investigated the bilinear operators defined by:
 \begin{equation}
B(f_1, f_2)(x) := \int_{\scriptsize{\R^2}}
m(\xi_1, \xi_2)
 \widehat{f_1}(\xi_1) 
\widehat{f_2}(\xi_2)
e^{2\pi i x(\xi_1 + \xi_2)}
d\xi_1
d\xi_2,
 \end{equation}\label{1}
where the multiplier $m$ satisfies the classical Marcinkiewicz-Mikhlin-H\"ormander condition
$$|\partial^{\alpha}(m(\xi))|\lesssim \frac{1}{|\xi|^{|\alpha|}}$$
for sufficiently many multi-indices $\alpha$ and where $\xi :=(\xi_1, \xi_2)$. Here, $\widehat{f}$ denotes the Fourier transform, which is defined by $ \widehat{f}(\eta) =\int_{\scriptsize{\R}} e^{-2\pi i x \eta} f(x)dx$, and $A\lesssim B$ denotes the assertion that $A\leq CB$ for some large constant $C$. Specifically, they obtained the following results:

\begin{theorem}
The operator $B$ maps $L^{p_1} \times L^{p_2} \rightarrow L^{p_3}$, provided $1<p_1,p_2\leq \infty$ and $0<p_3<\infty$, where $1/{p_1}+1/{p_2} =1/{p_3}$.
\end{theorem}

If $m$ is identically equal to one, then the operator $B$ is the pointwise product operator, so the wide range of $L^p$ estimates for the operators $B$ serves as a generalization of H\"older's inequality in which products are replaced by classical paraproducts, which satisfies the classical Marcinkiewicz-Mikhlin-H\"ormander condition. Theorem 1.1 is critical to understanding the general Korteweg-de Vries equation, which describes weakly non-linear shallow water waves, and the Leibnitz rule, which refers to inequalities of the type
 \begin{equation}
 \| D^\alpha(f g)\|_p \lesssim \| D^\alpha f\|_{p_1}  
 \|g\|_{q_1} 
 +\|  f\|_{p_2}
 \| D^\alpha g\|_{q_2}
 \end{equation}
 where $1<p_i,q_i \leq \infty, \; 1/{p_i}+1/{q_i} =1/p$ for $i=1,2$ and $1/{(1+\alpha)}<p<\infty$.


If the multiplier $m$ in (\ref{1}) is replaced with $\chi_{ \xi_1 <\xi_2}$, which has discontinuities along the line $\{\xi_1 =\xi_2\}$, one establishes the bilinear Hilbert transform(BHT) given by the equation
\begin{equation}\label{BHT}
BHT(f_1, f_2)(x) := \int_{\xi_1 <\xi_2}
 \widehat{f_1}(\xi_1) 
\widehat{f_2}(\xi_2)
e^{2\pi i x(\xi_1 + \xi_2)}
d\xi_1
d\xi_2, 
\end{equation}
modulo minor modification, where $f_1, f_2$ are test functions on $\R$.
The bilinear Hilbert transform was introduced by Calder\'on when he understood the first Calder\'on commutator in connection to the Cauchy integral on Lipschitz curves. Lacey and Thiele, \cite{LT1} in 1997 and \cite{LT2} in 1999,
 provide a wide range of $L^p$ estimates for the bilinear Hilbert transform. Specifically, they announced the following theorem: 
\begin{theorem}
The operator $BHT$ maps $L^{p_1} \times L^{p_2} \rightarrow L^{p_3}$, provided $1<p_1,p_2\leq \infty$ and $2/3<p_3<\infty$, where $1/{p_1}+1/{p_2} =1/{p_3}$.
\end{theorem}

Work of Muscalu, Tao, Thiele \cite{MTT2}  in 2001 provides H\"older type $L^p$ estimates on a trilinear variant of the bilinear Hilbert transform given by
 \begin{equation}\label{biest}
 T(f_1, f_2, f_3)(x) := \int_{\xi_1 <\xi_2<\xi_3}
 \widehat{f_1}(\xi_1) 
 \widehat{f_2}(\xi_2)
\widehat{f_3}(\xi_3)
e^{2\pi i x(\xi_1 + \xi_2+\xi_3)}
d\xi_1
d\xi_2
d\xi_3.
 \end{equation}
The operator $T$ may be considered to be the continuous analogue of an iterated Fourier series with a product of three functions. The operator $T$ has deep connections to AKNS systems (named after Ablowitz-Kaup-Newell-Segur), which describe various models in partial differential equations, including the Korteg-de Vries, and non-linear  Schor\"dinger equations. 

\vskip 1cm

 \section{Statement of Results}

We now construct trilinear operators $T_{m_1 m_2}$ with more generic symbols with the same singularity sets as the one of $T$ in (\ref{biest}).  We define such an operator by:
 \begin{equation}\label{mine}
 T_{m_1 m_2}(f_1, f_2,f_3)(x) := \int_{\scriptsize{\R^3}} m_1(\xi_1, \xi_2)m_2(\xi_2, \xi_3)
\widehat{f_1}(\xi_1)
\widehat{f_2}(\xi_2)
\widehat{f_3}(\xi_3)
e^{2\pi i x(\xi_1 + \xi_2+\xi_3)}
d\xi_1
d\xi_2
d\xi_3,
 \end{equation}
where $m_i(\xi_i, \xi_{i+1})$, for each $i=1,2$, is smooth away from the line
 \begin{equation}\label{Gamma}
\Gamma_i=\{(\xi_i, \xi_{i+1})\in \R^2 : \xi_i=\xi_{i+1}\}
\end{equation}
and satisfies the condition
$$|\partial^{\alpha}(m_i(\xi))|\lesssim \frac{1}{\dist(\Gamma_i,  \xi)^{|\alpha|}}$$
for every $\xi \in \R^2\setminus \Gamma_i$ and sufficiently many multi-indices $\alpha$. 

 If we consider the singular symbol $\chi_{\xi_1 <\xi_2<\xi_3}$ of $T$ as described in (\ref{biest}) as $\chi_{\xi_1 <\xi_2}\cdot\chi_{\xi_2<\xi_3}$, then it is easy to see that the symbol $ m_1(\xi_1, \xi_2)\cdot m_2(\xi_2, \xi_3)$  is a natural variant of $\chi_{\xi_1 <\xi_2}\cdot\chi_{\xi_2<\xi_3}$ with the same singular sets along two hyperplanes, $\xi_1=\xi_2$ and $ \xi_2=\xi_3$.  A simple example for this symbol is $a(\xi_1-\xi_2)\cdot b(\xi_2-\xi_3)$ for classical Marcinkiewicz-Mikhlin-H\"ormander symbols $a$ and $b$. If we replace both $ m_1(\xi_1, \xi_2)$ and $m_2(\xi_2, \xi_3)$ with classical symbols $a(\xi_1, \xi_2)$ and $b(\xi_2, \xi_3)$ satisfying the classical Marcinkiewicz-Mikhlin-H\"ormander condition, then 
 we obtain an example of flag paraproducts 
 $$T_{ab}(f_1, f_2, f_3)(x):= \int_{\scriptsize{\R^3}} a(\xi_1, \xi_2) b(\xi_2, \xi_3)
 \widehat{f_1}(\xi_1) 
 \widehat{f_2}(\xi_2)
\widehat{f_3}(\xi_3)
e^{2\pi i x(\xi_1 + \xi_2+\xi_3)}
d\xi_1
d\xi_2
d\xi_3,
$$ 
 which were introduced by Muscalu \cite{MC} in 2007, who proved a variety of H\"older type $L^p$ estimates for this family. Flag paraproducts arise naturally in many settings, including non-linear partial differential equations and probability theory. In fact, we may consider the operator $T_{m_1 m_2}$ as an infinite sum of flag paraproducts. This is why we need to introduce tiles and wave packets associated to tiles, which are discussed later, as is done in the proof in \cite{LT1} of $L^p$ estimates for the bilinear Hilbert transform, in order to overcome this complexity.
  Let us consider the 3-dimensional affine hyperspace
$$S := \{(\alpha_1, \alpha_2, \alpha_3, \alpha_4) \in \R^4 |\  \alpha_1 + \alpha_2 + \alpha_3 + \alpha_4 = 1\}.$$

We denote by $\\D'$ the open interior of the convex hull of the following twelve extremal points $A_1, \cdots, A_{12},$ which belong to $S$: 
\begin{align}
&A_1 =(1, 1/2, 1, -3/2) \; \;   A_2 =(1/2, 1, 1, -3/2) \; \;   A_3=(1/2,1,-3/2,1) \nonumber\\
&A_4=(1,1/2,-3/2, 1) \; \;  A_5=(1, -1/2, 0, 1/2) \; \;   A_6=(1, -1/2, 1/2, 0)\nonumber\\
& A_7=(1/2,-1/2,0,1)\; \;  A_8=(1/2,-1/2,1, 0) \; \;   A_9 = (-1/2, 1, 0, 1/2) \nonumber\\
&  A_{10} =(-1/2, 1, 1/2, 0)\; \;   A_{11}=(-1/2,1/2,1,0)\; \; A_{12}=(-1/2,1/2,0, 1).\nonumber
 \end{align}

In addition, we denote by $\\D''$ the open interior of the convex hull of the twelve extremal points, $\tilde{A}_1, \cdots, \tilde{A}_{12},$ in $S$, where the points $\tilde{A}_j$ are given by exchanging the 1st coordinate and the 3rd coordinate of $A_j$ for $j=1,\dots, 12$.
Then, we set $\\D := \\D' \cap \\D''$. 

\begin{theorem}
The operator $T$ maps $L^{p_1} \times L^{p_2} \times L^{p_3} \rightarrow L^{p_4}$, provided
$(1/{p_1}, 1/{p_2},1/{p_3},1-1/{p_4}) \in \\D$ where $1<p_1,p_2, p_3 \leq \infty$ and $0<p_4<\infty$.
\end{theorem}



 In particular, $T$ maps $L^{p_1} \times L^{p_2} \times L^{p_3} \rightarrow L^{p_4}$ as long as $1 < p_1, p_2, p_3 \leq \infty$ and $1\leq p_4 <\infty$.

One of the essential ideas in \cite{MTT2} is to consider the symbol 
$\chi_{\xi_1 <\xi_2}\cdot\chi_{\xi_2<\xi_3}$ as $\chi_{\xi_1 <\xi_2}\cdot\chi_{{\frac{\xi_1+\xi_2}{2}}<\xi_3}$
 in the region $\{|\xi_3- \xi_2| \gg |\xi_2- \xi_1|\}$,
  and, similarly, to consider
$\chi_{\xi_1 <\xi_2}\cdot\chi_{\xi_2<\xi_3}$ as $\chi_{\xi_1<{\frac{\xi_2+\xi_3}{2}}}\cdot \chi_{\xi_2 <\xi_3}$
 in the region  $\{|\xi_3- \xi_2| \ll |\xi_2- \xi_1|\}$. This observation enabled the authors to decompose the multiplier operator corresponding to each region with a composition of two bilinear Hilbert transforms with some constraint on the inner bilinear Hilbert transform. This decomposition enabled the authors to establish its $L^p$ estimates by applying the idea from the corresponding proof for the bilinear Hilbert transform. However, we cannot apply this technique to $ T_{m_1 m_2}$, because the symbols $m_1(\xi_1, \xi_2) \cdot m_2(\xi_2, \xi_3)$ are truly dependent on its variables, so we cannot rely on the benefit from the characteristic functions; This makes the whole proof here is more challenging.



 The main purpose of this paper is to obtain a large set of $L^p$ estimates for $T_{m_1 m_2}$. Specifically, we obtain the following result:

\begin{theorem}\label{main theorem}
The operator  $T_{m_1 m_2}$ maps 
\begin{equation}\label{main}
T_{m_1 m_2} :L^{p_1} \times L^{p_2} \times L^{p_3} \rightarrow L^{p_4},
\end{equation}
 provided
$(1/{p_1}, 1/{p_2},1/{p_3},1-1/{p_4}) \in \\D$ where $1<p_1,p_2, p_3 \leq \infty$, $0<p_4<\infty$, and $1/{p_1}+1/{p_2}+1/{p_3} =1/{p_4}$.
\end{theorem}




The notation and techniques used here to prove Theorem 2.2 are inspired by the work of Muscalu, Tao and Thiele in \cite{MTT1} and \cite{MTT2}.

\vskip .5cm


\textbf{ Acknowledgements:}
The author would like to express thanks to her dissertation adviser, Camil Muscalu, for valuable conversations
and his guidance regarding this paper.
\vskip 1cm
\section{Restricted Weak-type Interpolation}
In this section, we will review the restricted weak-type interpolation theorems from \cite{MTT4}. These theorems allow us to reduce multilinear $L^p$ estimates, such as those in Theorem 2.2, to certain restricted weak-type estimates.


 To prove the $L^p$ estimates on $T_{m_1m_2}$ it is convenient to use duality to convert (\ref{main})
 from a trilinear operator estimate to a quadrilinear form estimate because this makes the estimate more
symmetric. We shall investigate the following quadrilinear form $\Lambda$ associated to $T_{m_1m_2}$:
 $$\Lambda(f_1, f_2, f_3, f_4) :=\int_{\scriptsize{\R}} T_{m_1m_2}(f_1, f_2, f_3)(x)f_4(x)dx.$$
The assertion that the operator $T_{m_1m_2}$ maps from $L^{p_1} \times L^{p_2} \times L^{p_3}$ to $L^{{p_4}}$ is equivalent to the assertion that $\Lambda$ is bounded on $L^{p_1} \times L^{p_2} \times L^{p_3}\times L^{p'_4}$, for $1<{p_4}<\infty$, where $1/{p_4}+1/{p'_4} =1$. For ${p_4}\leq1$, this simple duality relationship breaks down, but the interpolation arguments in \cite{MTT4} will enable us to reduce the statement in $(\ref{main})$ to certain restricted type estimates on $\Lambda.$ As in \cite{MTT4}, we find it more convenient to work with the quantities $\alpha_i = 1/{p_i}$, $i = 1, 2, 3$ and $\alpha_4 = 1/{p'_4}$, where $p_i$ stands for the exponent of $L^{p_i}$.
 
\begin{definition} A tuple $\alpha= (\alpha_1, \alpha_2, \alpha_3, \alpha_4)$ with $-\infty<\alpha_i <1$, for all $1 \leq i \leq 4$, is called admissible if  $\sum_{i=1}^4 \alpha_i =1$
 and there is at most one index $j$ such that $\alpha_j < 0$. We call an index $i$ good if $\alpha_i \geq 0$, and we call it bad if $\alpha_i < 0$. A good tuple is an admissible tuple which has no bad indices and a bad tuple
is an admissible tuple with one bad index.
\end{definition}

\begin{definition}Let $E$ be a subset $E \subset \R$ with finite measure and $E'$ be a subset of $E$. Define that $E'$ is a major subset of $E$, provided $2|E'| \geq  |E|$.
\end{definition}

\begin{definition}
Let $E$ be a subset $E \subset \R$ with finite measure. $X(E)$ denotes the space of all functions $f$ with $|f| \leq \chi_E$ almost everywhere.
\end{definition}

\begin{definition}
Fix an admissible tuple $\alpha=(\alpha_1, \alpha_2, \alpha_3,\alpha_4)$.
A quadrilinear form $\Lambda$ is of restricted type $\alpha$ if for each tuple $(E_1,\, E_2,\,E_3,\,E_4)$ of subsets of $\R$ with finite
measure, there exists a constant C and a major subset $E'_j$ of $E_j$, for each (one or none) bad index $j$, such
that 
$$|\Lambda(f_1, f_2, f_3, f_4)|\leq C |E_1|^{\alpha_1} |E_2|^{\alpha_2}|E_3|^{\alpha_3}|E_4|^{\alpha_4}$$
for each tuple $(f_1, f_2, f_3, f_4)$ of functions with $f_j \in X(E'_j)$ when $j$ is a bad index (if exists) and $f_i \in X(E_i)$ when $i$ is a good index. 
\end{definition}

The following restricted type result will be proved in the last two sections:
\begin{theorem} 
For every vertex $A_i, \tilde{A}_i, \, i = 1, . . . , 12$, there exists an admissible tuple $\alpha$ arbitrarily close to the vertex such that the form $\Lambda$ is of restricted type $\alpha$.
\end{theorem}

By interpolation of restricted weak-type estimates in \cite{MTT4},  we thus obtain that Theorem 3.5 implies that $\Lambda$ is of restricted type $\alpha$ for any admissible tuples $\alpha \in \\D$.



It only remains to convert these restricted type estimates into strong type estimates
$(\ref{main})$. To do this, one can just apply (exactly as in \cite{MTT4}) the multilinear Marcinkiewicz
interpolation theorem in \cite{JS} for good tuples and the interpolation lemma 3.11 in
\cite{MTT4} for bad tuples.
This completes the proof of Theorem 2.2. Therefore, it remains to prove Theorem 3.5 here.

\vskip 1cm


\section{Discretization}
In order to prove Theorem 3.5, we first discretize the continuous form $\Lambda$ and reduce the relevant result to estimates for a discretized variant involving sums of inner products with wave packets.
We begin this section by recalling some standard definitions from \cite{MTT2}.
 
  For any interval $I$, we denote by $|I|$ the Lebesgue measure of $I$ and by $cI$ with $c>0$ the interval with the same center as $I$ but $c$ times the side-length. For any (quasi) cube and square $Q$, we denote by $|Q|$ its side-length and by $cQ$ with $c>0$ the cube with the same center as $Q$ but with the side-length $c|Q|$.


 \begin{definition} Fix a positive integer $n$ and a n-tuple $\sigma \in \{0, \frac 1 3 , \frac 2 3 \}^n$. Define the shifted n-dyadic grid $D=D_\sigma^n$ to be the collection of cubes of the form
 $$D_\sigma^n :=\{2^j(k+(0,1)^n +(-1)^j\sigma)| j\in \Z, k\in \Z^n\}.$$
 \end{definition}
 
 \begin{definition} 
 A subset $D'$ of a shifted n-dyadic grid $D$ is called sparse if it satisfies the following properties, for any two cubes $Q, Q'$ in $D$ with $Q \neq Q'$:
  \begin{enumerate}
 \item $|Q|<|Q'|$ implies $10^9|Q|<|Q'|.$
\item $ |Q|=|Q'|$ implies $10^9Q \cap 10^9Q'= \emptyset.$
\end{enumerate}
 \end{definition}
 
 We observe that any subset of a shifted n-dyadic grid with $n \leq 3$ can be split into
$O(1)$ sparse subsets.
 
  \begin{definition} Let a 3-tuple $\sigma \in \{0, \frac 1 3 ,\frac 2 3\}^3$ be a shift.  A collection $\Q \subset D^3_\sigma$ of cubes is said to have rank $1$ if it satisfies the following properties for all $Q,Q' \in \Q$:
  \begin{enumerate}
\item If $Q \neq Q'$, then $Q \cap Q' = \emptyset$.
\item If $Q \neq Q'$, then $Q_i \neq Q'_i$ for all $i = 1, 2, 3$.
\item If $3Q'_i \subset 3Q_i$ for some $i=1,2,3,$ then $10^7Q'_j \subset 10^7Q_j$ for all $j=1,2,3$.
\item If $|Q'| < 10^9|Q|$ and $3Q'_i \subset 3Q_i$ for some $i=1,2,3,$ then $3Q'_j \cap  3Q_j = \emptyset$ for all $j \neq i$.
  \end{enumerate}
 \end{definition}

  \begin{definition} Fix $\sigma = (\sigma_1, \sigma_2, \sigma_3) \in \{0, \frac1 3 , \frac 2 3\}^3$ and fix $1 \leq i \leq 3$. An i-tile with shift $\sigma_i$
is a rectangle $P = I_P \times \omega_P$ in the phase plane such that $I_P \in D_0^1,\; \omega_P \in D_{\sigma_i}^1$, and $|I_P|\cdot |\omega_P|=1$. A tri-tile with shift $\sigma$ is a 3-tuple $\vec P = (P_1, P_2, P_3)$ such that each $P_i$ is an i-tile with shift $\sigma_i$, and each $I_{P_i}$ is independent of $i$, which we denote $I_{\vec P}$. The frequency cube $Q_{\vec P}$ of a tri-tile $\vec P$ is defined by $\prod_{i=1}^3 \omega_{P_i} $.
 \end{definition}
 
  We shall sometimes refer to $i$-tiles with shift $\sigma_i$ just as $i$-tiles, or even as tiles, if the
parameters $\sigma_i$, $i$ are insignificant.
 
  \begin{definition}A set $\vec \P$ of tri-tiles is called sparse, if all tri-tiles in $\vec \P$ have the same shift and the set of the frequency cube $Q_{\vec P}$ with $\vec P\in \vec\P$ is sparse.
 \end{definition}
 Similarily, observe that any set of tri-tiles can be split into $O(1)$ sparse subsets.
  \begin{definition} Let $P$ and $P'$ be tiles. We define the following notations:
 \begin{enumerate} 
  \item $P' < P$ if $I_{P'}\subsetneq  I_P$ and $3\omega_P \subseteq 3 \omega_{P'}$.
  \item $P' \leq P$ if $P' < P$ or $P' = P$.
  \item $P'\lesssim P$  if $I_{P'} \subseteq I_P$ and $10^7\omega_P \subseteq 10^7\omega_{P'}$. 
  \item $P' \lesssim' P$ if $P'\lesssim P$ and $P' \nleq P$.
\end{enumerate} 
 \end{definition}
 
 
 
 \begin{definition} A collection $\vec\P$ of tri-tiles is said to have rank $1$ if it satisfies the following properties for all $\vec P, \vec P'  \in \vec\P$:
\begin{enumerate} 
\item If $\vec P \neq \vec P'$, then $P_i \neq P'_i$ for all $i = 1, 2, 3$.
\item If $P'_i \leq P_i$ for some $i=1,2,3,$ then  $P'_j \lesssim P_j$ for all $j=1,2,3$.
\item f $P'_i \leq P_i$ for some $i=1,2,3$ and furthermore $| I_{\vec {P'}}| < 10^9| I_{\vec {P}}|$, then we have $P'_j \lesssim' P_j$ for all $j \neq i$. 
\end{enumerate}
 \end{definition}
\begin{definition} Let $P$ be a tile. A wave packet $\Phi_P$ adapted to $P$ is a function such that $\widehat{\Phi}_{P}$ is supported in $\frac 9 {10} \omega_P$ and 
\begin{equation}\label{kk}
|\Phi_P(x)| \lesssim |I_P|^{-1/2} \widetilde{\chi}_{I_P}(x)^M
\end{equation}
for all $M > 0$, with the implicit constant depending on $M$. Here, $\widetilde{\chi}_{I_P}$ denotes the approximate cutoff function of the interval ${I_P}$, which is defined by
 $$\widetilde{\chi}_{I_P}(x) := (1+ ( \frac {|x-x_{I_P}|}{|{I_P}|})^2)^{-1/2},$$ where $x_{I_P}$ is the center of the interval ${I_P}$.
  \end{definition}

 The discretized variant of Theorem 3.5 is as follows.

\begin{theorem} Let $\sigma, \sigma' \in \{0, \frac 1 3 ,\frac 2 3\}^3$ be shifts, and let $\vec\P, \vec\Q$ be finite collections of tri-tiles with shifts $\sigma, \sigma'$, respectively, such that both $\vec\P$ and $\vec\Q$ have rank $1$. For each $\vec P \in \vec\P$, and $\vec Q \in \vec\Q$, we let  $\Phi_{P_i} = \Phi_{{P_i},i}$ and $\Phi_{Q_i} = \Phi_{{Q_i},i}$ be wave packets on $P_i$ and $Q_i$, $i = 1, 2, 3$, respectively. Define the forms $\Lambda_{ \vec\P,\vec\Q}$ and $\Lambda^{\#}_{ \vec\P,\vec\Q}$ by 
 $$\Lambda^{\#}_{ \vec\P,\vec\Q}(f_1, f_2, f_3, f_4) := \sum_{\vec P \in \vec\P} \frac {1}{|I_{\vec P}|^{1/2}}
 \langle B^{\#}_{P_1}(f_1, f_2),  \Phi_{P_1}\rangle
 \langle f_3 , \Phi_{P_2}\rangle
       \langle f_4 , \Phi_{P_3}\rangle$$
 where
  $$B^{\#}_{P_1}(f_1, f_2):= \sum_{\vec Q \in \vec\Q: \atop \omega_{Q_3}\subset \omega_{P_1}, \;2^{\#}|\omega_{Q_3}| \sim |\omega_{P_1}|}
\frac {1} {|I_{\vec Q}|^{1/2}}
\langle f_1, \Phi_{Q_1}\rangle 
   \langle f_2 , \Phi_{Q_2}\rangle
   \Phi_{Q_3}$$
   and
   $$\Lambda_{ \vec\P,\vec\Q}(f_1, f_2, f_3, f_4)
:= \sum_{\l=1}^{M}
\sum_{\#\geq 1000}
2^{-(\#+1)\l}
\Lambda^{\#}_{ \vec\P,\vec\Q}(f_1, f_2, f_3, f_4)$$
 for $\# \geq 1000$ and for a sufficiently big number $M\in \Z$. Then $\Lambda^{\#}_{ \vec\P,\vec\Q}$ is of restricted type $\alpha$, modulo extra factor $2^{\#/2}$,
 for all admissible tuples $\alpha \in \\D$, uniformly in the parameters $\sigma, \sigma', \vec \P, \vec \Q, \Phi_{P_i}, \Phi_{Q_i}$. Therefore, for any positive integer $M$, $\Lambda_{ \vec\P,\vec\Q}$ is of restricted type $\alpha$ for all admissible tuples $\alpha \in \\D$, uniformly in the parameters $ \sigma, \sigma', \vec \P, \vec \Q, \Phi_{P_i}, \Phi_{Q_i}$. Furthermore, in the case that $\alpha$ has a bad index $j$, the major subset $E'_j$ can be chosen independently of the parameters $\sigma, \sigma', \vec \P, \vec \Q, \Phi_{P_i}, \Phi_{Q_i}$.
 
 \end{theorem}

 Note that once we obtain the first conclusion in Theorem 4.9, say
 \begin{equation}\label{first}
 |\Lambda^{\#}_{ \vec\P,\vec\Q}(f_1, f_2, f_3, f_4)| \lesssim 2^{\#/2}|E_1|^{\alpha_1} |E_2|^{\alpha_2}|E_3|^{\alpha_3}|E_4|^{\alpha_4} 
 \end{equation}
 for given tuple $(E_1,\, E_2,\,E_3,\,E_4)$ of subsets of $\R$ with finite
measure and for each tuple $(f_1, f_2, f_3, f_4)$ of functions with $f_j \in X(E'_j)$ when $j$ (if exists) is a bad index, and $f_i \in X(E_i)$ when $i$ is a good index, then we obtain the following conclusion:
 \begin{align}
 |\Lambda_{ \vec\P,\vec\Q}(f_1, f_2, f_3, f_4)|
&\leq 
\sum_{\l=1}^{M}
\sum_{\#\geq 1000}
2^{-(\#+1)\l}
|\Lambda^{\#}_{ \vec\P,\vec\Q}(f_1, f_2, f_3, f_4)|\nonumber\\
&\lesssim
\sum_{\l=1}^{M}
\sum_{\#\geq 1000}
2^{-(\#+1)\l}
2^{\#/2}|E_1|^{\alpha_1} |E_2|^{\alpha_2}|E_3|^{\alpha_3}|E_4|^{\alpha_4}\nonumber\\
&\leq
|E_1|^{\alpha_1} |E_2|^{\alpha_2}|E_3|^{\alpha_3}|E_4|^{\alpha_4}\nonumber
\end{align}
for any positive integer $M$. Henceforth, we shall focus on showing the first conclusion.

 \vskip .7cm

 The remainder for this section is devoted to showing how Theorem 3.5 can be deduced from Theorem 4.9.
 First, we divide the symbol  $m_1 m_2$ in (\ref{mine}) as 
\begin{align}
m_1(\xi_1, \xi_2)&=\sum_{Q} m_1(\xi_1, \xi_2)\phi_Q(\xi_1, \xi_2),\nonumber\\
m_2(\xi_2, \xi_3)&=\sum_{Q'} m_2(\xi_2, \xi_3) \phi_{{Q'}}(\xi_2, \xi_3)\nonumber
\end{align}
by a standard partition of unity, where $Q, {{Q'}}$ are the shifted dyadic squares such that $Q=Q_1 \times Q_2$ and $Q'=Q'_1 \times Q'_2$. Here, $Q_j, Q'_j$ are shifted dyadic intervals $2^k (n+(1, 0)+\sigma)$, for any $k, n \in \Z$ and $\sigma \in \{0, 1/3, 2/3\}$, satisfying the property 
\begin{align}
\dist(Q, \Gamma_1) &\simeq C_0 |Q|\nonumber\\
\dist(Q', \Gamma_2) &\simeq C'_0 |Q'|\nonumber,
\end{align}
where $C_0, C'_0$ are fixed large constants and $\Gamma_j$, for each $j=1,2$, is defined by (\ref{Gamma}). Here, each $ \phi_Q, \phi_Q'$ are smooth bumps adapted to $Q, Q'$ and supported in $\frac 8 {10} Q, \frac 8 {10} Q',$ respectively. Then, we obtain
\begin{align}
m_1(\xi_1, \xi_2)m_2(\xi_2, \xi_3)&=\sum_{Q, {{Q'}}}  m_1(\xi_1, \xi_2) \phi_Q(\xi_1, \xi_2) \cdot m_2(\xi_2, \xi_3)\phi_{{Q'}}(\xi_2, \xi_3)\nonumber\\
&=\sum_{|Q|\ll|{Q'}|} +\sum_{|Q|\sim|{Q'}|}+\sum_{|Q|\gg|{Q'}|}:= m_{\I}+m_{\I\I}+m_{\I\I\I},\nonumber
\end{align}
which is essentially similar to looking at the symbol separately in the three regions $\chi_{|\xi_2-\xi_1|\ll |\xi_3-\xi_2|}$, $\chi_{|\xi_2-\xi_1|\sim |\xi_3-\xi_2|}$ and $\chi_{|\xi_2-\xi_1|\gg |\xi_3-\xi_2|}$. 

The multiplier operator $T_{\I\I}$ defined by
\begin{align*}
T_{\I\I}(f_1, f_2,f_3)(x) := \int_{\scriptsize{\R^3}} m_{\I\I}(\xi_1,\xi_2, \xi_3)
\widehat{f_1}(\xi_1)
\widehat{f_2}(\xi_2)
\widehat{f_3}(\xi_3)
e^{2\pi i x(\xi_1 + \xi_2+\xi_3)}
d\xi_1
d\xi_2
d\xi_3
\end{align*}
is similar to the Bilinear Hilbert Transform with three functions. Specifically, it satisfies
$$|\partial^{\alpha}(m_{\I\I}(\xi))|\lesssim \frac{1}{\dist(\Gamma, \xi)^{|\alpha|}}$$
for every $\xi \in \R^3\setminus \Gamma$ and sufficiently many multi-indices $\alpha$, where $\Gamma$ is given by the line $\Gamma:=\{(\xi_1, \xi_2, \xi_3)\in \R^3 : \xi_1=\xi_2=\xi_3\}$. In fact, such operators were studied in \cite{MTT4} and Muscalu, Tao and Thiele established a wider range of $L^p$ estimates for them than those in Theorem 2.2.

Now we claim that the multiplier operator  $T_{\I}$ given by
\begin{align*}
T_{\I}(f_1, f_2,f_3)(x) := \int_{\scriptsize{\R^3}} m_{\I}(\xi_1,\xi_2, \xi_3)
\widehat{f_1}(\xi_1)
\widehat{f_2}(\xi_2)
\widehat{f_3}(\xi_3)
e^{2\pi i x(\xi_1 + \xi_2+\xi_3)}
d\xi_1
d\xi_2
d\xi_3
\end{align*}
maps $L^{p_1} \times L^{p_2} \times L^{p_3} \rightarrow L^{p_4}$, provided
$(1/{p_1}, 1/{p_2},1/{p_3},1-1/{p_4}) \in \\D'$, where $1<p_1,p_2, p_3 \leq \infty$, $0<p_4<\infty$, and $1/{p_1}+1/{p_2}+1/{p_3} =1/{p_4}$.  And similarly, claim that the operator $T_{{\I\I\I}}$ with the symbol $m_{\I\I\I}$ has the same $L^p$ estimates as long as $(1/{p_1}, 1/{p_2},1/{p_3},1-1/{p_4}) \in \\D''.$ If these claims for $T_{\I}$ and $T_{\I\I\I}$ hold, then we finally obtain Theorem 2.2. We shall only prove the claim for  $T_{\I}$, as the claim for $T_{\I\I\I}$ follows by a permutation of the indices 1 and 3.

Consider 
\begin{align}
m_\I&=\sum_{|Q|\ll|{Q'}|} m_1(\xi_1, \xi_2) \phi_Q(\xi_1, \xi_2) \cdot m_2(\xi_2, \xi_3) \phi_{{Q'}}(\xi_2, \xi_3)\nonumber\\
&=\sum_{Q} m_1(\xi_1, \xi_2)\phi_Q(\xi_1, \xi_2) \left(\sum_{{{Q';}} |Q|\ll|{Q'}|}m_2(\xi_2, \xi_3)\phi_{{Q'}}(\xi_2, \xi_3)\right)\nonumber\\
&=\sum_{Q} m_1(\xi_1, \xi_2) \phi_Q(\xi_1, \xi_2) \left(\sum_{{{Q';}} |Q|\ll|{Q'}|} \sum_{\n \in  \scriptsize{\Z^2}} C_{\n}^{{{Q'}}} \phi_{Q'_1, \n,1}(\xi_2)\phi_{Q'_2, \n,2}(\xi_3)\right).\label{5}
\end{align}
We establish the last equality by computing a double Fourier series of $m_2(\xi_2, \xi_3)\phi_{{Q'}}(\xi_2, \xi_3)$ where $\phi_{Q'_i,\n,i}$ is a bump function adapted to $Q'_i$ and supported to $\frac 9 {10}Q'_i$ uniformly in $n_i$ for $(n_1, n_2):=\n \in \Z^2$. Here, the Fourier coefficient is given by
\begin{equation}\label{fc1}
C_{\n}^{{{Q'}}} = \frac 1 {|{Q'}|^2} \int_{\scriptsize{\R^2}} m_2(\xi_2, \xi_3) \phi_{{Q'}}(\xi_2, \xi_3) e^{-2\pi i n_1 {\xi_2 \over {|Q'|}}} e^{-2\pi i n_2 {\xi_3 \over {|Q'|}}}d\xi_2 d\xi_3.
\end{equation}
We now assert that $|C_{\n}^{{{Q'}}}| \lesssim C(\n)$, where the implicit constant does not depend on each $Q'$, and $C(\n)$ is a rapidly decreasing sequence. This assertion is justified later in Section 5.

Then, we can majorize $(\ref{5})$ by
\begin{equation}\label{6}
\sum_{\n \in  \scriptsize{\Z^2}} C(\n)\sum_{Q, {{Q'}}; |Q|\ll|{Q'}| } m_1(\xi_1, \xi_2) \phi_Q(\xi_1, \xi_2) \phi_{Q'_1, \n,1}(\xi_2)\phi_{Q'_2, \n,2}(\xi_3).
\end{equation}
Because of big decaying factor $C(\n)$, once we have $L^p$ estimates of the corresponding trilinear operator with the symbol 
\begin{equation}\label{symbol1212}
\sum_{Q, {{Q'}}; |Q|\ll|{Q'}| } m_1(\xi_1, \xi_2)\phi_Q(\xi_1, \xi_2)  \phi_{Q'_1, \n,1}(\xi_2)\phi_{Q'_2, \n,2}(\xi_3),
\end{equation}
then we can control the last summation with respect to $\n\in \Z^2$.

Now by applying the Taylor series to $\phi_{Q'_1, \n,1}(\xi_2)$, we obtain that
\begin{equation}\nonumber
 \phi_{Q'_1, \n,1}(\xi_2)=\sum_{\l=0}^M \phi_{Q'_1, \n,1}^{(\l)}\left(\frac{\xi_1+\xi_2}{2}\right)\left(\frac{\xi_2-\xi_1}{2}\right)^\l \frac 1 {\l!} +R_M(\xi_1, \xi_2) 
\end{equation}
 for a sufficiently big number $M\in \Z$. Here $R_M(\xi_1, \xi_2)$ is the remainder term given by
 \begin{equation}\label{remainder}
 R_M(\xi_1, \xi_2)= \frac 1 {M!} \phi_{Q'_1, \n,1}^{(M)}(\xi_{\theta})\left(\frac{\xi_2-\xi_1}{2}\right)^{M},
 \end{equation}
  where $\xi_{\theta} =(1-\frac \theta 2)\xi_2 +\frac \theta 2 \xi_1$ for some $\theta$ between $0$ and $1$.

We split (\ref{symbol1212}) into three terms with $\l=0$, with $1\leq \l< M$, and with the remainder term $R_M(\xi_1, \xi_2)$ as following:
\begin{align}
& \sum_{Q, {{Q'}}; |Q|\ll|{Q'}| } 
 m_1(\xi_1, \xi_2)\phi_Q(\xi_1, \xi_2)\phi_{Q'_1, \n,1}\left(\frac{\xi_1+\xi_2}{2}\right)\phi_{Q'_2, \n,2}(\xi_3)\nonumber\\
&+\sum_{\l=1}^{M-1}\sum_{Q, {{Q'}}; |Q|\ll|{Q'}| } m_1(\xi_1, \xi_2) \phi_Q(\xi_1, \xi_2)  \phi_{Q'_1, \n,1}^{(\l)}\left(\frac{\xi_1+\xi_2}{2}\right)\left(\frac{\xi_2-\xi_1}{2}\right)^\l \frac 1 {\l!} \phi_{Q'_2, \n,2}(\xi_3)\nonumber\\
&+\sum_{Q, {{Q'}}; |Q|\ll|{Q'}| } m_1(\xi_1, \xi_2) \phi_Q(\xi_1, \xi_2) R_M(\xi_1, \xi_2)\phi_{Q'_2, \n,2}(\xi_3) \nonumber\\
&:=m_{\I, \{\l=0\}}+\sum_{\l=1}^{M-1}    m_{\I, \l}+m_{\I, R_M}.\label{8}
\end{align}

 We can easily check that $m_{\I, \{\l=0\}}$ is similar to the symbol of the operator $T$ in (\ref{biest}) after splitting $m_1(\xi_1, \xi_2)\phi_Q(\xi_1, \xi_2)$ as a double Fourier series in $\xi_1, \xi_2$
(modulo extra outer summation and rapidly decaying factor from the Fourier coefficients, like $C(\n)$ in (\ref{6})). Thus, we now consider the other two cases in (\ref{8}) more carefully.

For $1\leq \l < M$, we rewrite  $m_{\I, \l}$ as
 \begin{align}\label{9090}
\sum_{\#\geq 1000}
\sum_{Q, {{Q'}}; \atop k_2-k_1=\# }  
 m_1(\xi_1, \xi_2) \phi_Q(\xi_1, \xi_2)
 \phi_{Q'_1, \n,1}^{(\l)}\left(\frac{\xi_1+\xi_2}{2}\right)
 \left(\frac{\xi_2-\xi_1}{2}\right)^\l \frac 1 {\l!} \phi_{Q'_2, \n,2}(\xi_3)
 \end{align}
 by letting $|{Q}|=2^{k_1}$ and $|{{Q'}}|=2^{k_2}$, for $k_1, k_2 \in \Z$, and by assuming that $k_2 = k_1 + \#$, with $\# \geq 1000$ as $k_1 \ll k_2$. Then (\ref{9090}) is equal to
 \begin{align}
  \sum_{\#\geq 1000}
\sum_{Q, {{Q'}}; \atop k_2-k_1=\# } 
2^{(k_1-1)\l} \left[ m_1(\xi_1, \xi_2) \phi_Q(\xi_1, \xi_2)\left(\frac{\xi_2-\xi_1}{2^{k_1}}\right)^\l \frac 1 {\l!} \right]
\phi_{Q'_1, \n,1}^{(\l)}\left(\frac{\xi_1+\xi_2}{2}\right) \phi_{Q'_2, \n,2}(\xi_3)\nonumber
\end{align}
\begin{align}
 =\sum_{\#\geq 1000}
\sum_{Q, {{Q'}}; \atop 2^\#{|Q|} \sim |Q'|} 
\frac {2^{(k_1-1)\l}}{2^{{k_2} \l}} \left[ m_1(\xi_1, \xi_2) \phi_Q(\xi_1, \xi_2)\left(\frac{\xi_2-\xi_1}{2^{k_1}}\right)^\l \frac 1 {\l!} \right]
  2^{{k_2} \l} \phi_{Q'_1, \n,1}^{(\l)}\left(\frac{\xi_1+\xi_2}{2}\right)& \phi_{Q'_2, \n,2}(\xi_3)\nonumber
\end{align}
 \begin{align}
  =\sum_{\#\geq 1000}
2^{-(\#+1)\l}
\sum_{Q, {{Q'}}; \atop 2^\#{|Q|} \sim |Q'|} 
\left[ \sum_{\s \in \scriptsize{\Z^2}} C_{\s}^{{Q,\l}} \phi_{Q_1, \s,1}(\xi_1)\phi_{Q_2, \s,2}(\xi_2) \right] 
2^{{k_2} \l} \phi_{Q'_1, \n,1}^{(\l)}\left(\frac{\xi_1+\xi_2}{2}\right) &\phi_{Q'_2, \n,2}(\xi_3)\nonumber
\end{align}

by computing a double Fourier series of $m_1(\xi_1, \xi_2) \phi_Q(\xi_1, \xi_2)\left(\frac{\xi_2-\xi_1}{2^{k_1}}\right)^\l \frac 1 {\l!}$, where $\phi_{{Q}_i, \s,i}$ is a bump function adapted to ${Q}_i$ and supported to $\frac 9 {10}{Q}_i$ uniformly in $s_i$ for $(s_1, s_2):=\s \in \Z^2$. Here, the Fourier coefficient $C_{\s}^{Q,\l}$ is given by
\begin{equation}\label{fc2}
C_{\s}^{Q,\l} = \frac 1 {2^{2k_1}} \int_{\scriptsize{\R^2}} m_1(\xi_1, \xi_2) \phi_Q(\xi_1, \xi_2)\left(\frac{\xi_2-\xi_1}{2^{k_1}}\right)^\l \frac 1 {\l!}  e^{-2\pi i s_1 {\xi_1 \over 2^{k_1}}} e^{-2\pi i s_2 {\xi_2 \over 2^{k_1}}}d\xi_1 d\xi_2.
\end{equation}
 We now assert that $|C_{\s}^{Q,\l}| \lesssim C(\s)$, where the implicit constant does not depend on $Q$ and $\l$, and $C(\s)$ is a rapidly decreasing sequence. This assertion is justified later in Section 5. Then, we can majorize $m_{\I, \l}$ by
\begin{align}
 \sum_{\s \in \scriptsize{\Z^2}} C(\s)
\sum_{\#\geq 1000}
2^{-(\#+1)\l}
\sum_{Q, {{Q'}}; \atop 2^\#{|Q|} \sim |Q'|
 } 
\phi_{Q_1, \s,1}(\xi_1)\phi_{Q_2, \s,2}(\xi_2) 
2^{{k_2} \l}& \phi_{Q'_1, \n,1}^{(\l)}\left(\frac{\xi_1+\xi_2}{2}\right) \phi_{Q'_2, \n,2}(\xi_3)\label{cm}
\end{align}
Here again, because of big decaying factor $C(\s)$, once we have $L^p$ estimates of the 
corresponding trilinear operator with the symbol  
\begin{equation}\label{9}
\sum_{\#\geq 1000}
2^{-(\#+1)\l}
\sum_{Q, {{Q'}}; \atop 2^\#{|Q|} \sim |Q'|
 } 
\phi_{Q_1, \s,1}(\xi_1)\phi_{Q_2, \s,2}(\xi_2)  2^{{k_2} \l} \phi_{Q'_1, \n,1}^{(\l)}\left(\frac{\xi_1+\xi_2}{2}\right) \phi_{Q'_2, \n,2}(\xi_3),
\end{equation}
then we can control the last summation with respect to $\s\in \Z^2$.

 We now claim that $ 2^{{k_2} \l} \phi_{Q'_1, \n,1}^{(\l)}(\cdot) $ is also a bump function adapted to a dyadic interval ${Q'_1}$. Observe that if a dyadic interval ${Q'_1}$ is given by ${2^{k_2}}[m, m+1]$, for some $m \in \Z$, then we can denote the bump function $\phi_{Q'_1, \n,1}(\xi)$ adapted to ${Q'_1}$ in (\ref{9}) by
 \begin{equation}\label{515}
  \phi_{Q'_1, \n,1}(\xi):=\phi\left(\frac \xi {2^{k_2}} -m\right)\cdot e^{2\pi i n_1 {\xi \over 2^{k_2}}},
  \end{equation}
  where $\phi$ is a bump function adapted to $[0, 1]$, and then (\ref{515}) is equal to
  $$\phi\left(\frac \xi {2^{k_2}} -m\right)\cdot e^{2\pi i n_1 ({\xi \over 2^{k_2}}-m)}= {\phi_{n_1}}\left(\frac \xi {2^{k_2}} -m\right),$$
 where ${\phi_{n_1}}(\xi)$ is given by $\phi(\xi)e^{2\pi i n_1 \xi}$, which is a bump function adapted to $[0, 1].$
 Then, we have
  \begin{equation}\label{11}
  2^{{k_2} \l} \phi_{Q'_1, \n,1}^{(\l)}(\xi) =\phi_{n_1}^{(\l)}\left(\frac \xi {2^{k_2}}-m\right),
  \end{equation}
  which is a bump function adapted to ${Q'_1}={2^{k_2}}[m, m+1]$, for all $1 \leq \l < M$, so that we denote  $2^{{k_2} \l} \phi_{Q'_1, \n,1}^{(\l)}(\xi)$ by $ \widetilde{\phi}_{{Q'_1}, \n, 1}(\xi),$ since each number $\l$ is unimportant here.
  
  Thus, (\ref{9}) can be written as
  $$\sum_{\#\geq 1000}
2^{-(\#+1)\l}
\sum_{Q, {{Q'}}; \atop 2^\#{|Q|} \sim |Q'|
 } 
\phi_{Q_1, \s,1}(\xi_1)\phi_{Q_2, \s ,2}(\xi_2) \widetilde{\phi}_{{Q'_1}, \n,1}  \left(\frac{\xi_1+\xi_2}{2}\right) \phi_{Q'_2, \n,2}(\xi_3)$$

\begin{equation}\label{12}
=\sum_{\#\geq 1000}
2^{-(\#+1)\l}
\sum_{Q, {{Q'}}; \atop 2^\#{|Q|} \sim |Q'|
 } 
\phi_{Q_1, \s,1}(\xi_1)\phi_{Q_2, \s,2}(\xi_2) \widetilde{\phi}_{{Q''_1}, \n, 1} ({\xi_1+\xi_2}) \phi_{Q'_2,\n,2}(\xi_3),
\end{equation}
where ${Q''_1}={2^{k_2 +1}}[m, m+1].$ Henceforth, we shall redefine ${Q'_1}$ to be ${Q''_1}$ and redefine $ \widetilde{\phi}_{{Q'_1}, \n, 1}$ accordingly for simplicity. 

In the last expression (\ref{12}), we see that
$\xi_1 \in \frac 9 {10} Q_1$ and $\xi_2 \in \frac 9 {10} Q_2$, which follows that 
$\xi_1 + \xi_2 \in \frac 9 {10} Q_1+ \frac 9 {10} Q_2$. As a consequence, one can find a shifted dyadic interval $Q_3$ with the properties that
$\frac 9 {10} Q_1+\frac 9 {10} Q_2 \subseteq \frac 7 {10} Q_3$
and 
$|Q_1|=|Q_2| \simeq |Q_3|$. In particular, there exists bump functions $\phi_{Q_3,\n,3}$ adapted to $Q_3$ uniformly in $\n \in \Z^2$ and supported in $\frac 9 {10} Q_3$, such that $\phi_{Q_3,\n,3}\equiv1$ on $\frac 9 {10} Q_1+\frac 9 {10} Q_2$. 

 Similarly, we can find a shifted dyadic interval $Q'_3$ satisfying $\frac 9 {10} Q'_1+\frac 9 {10} Q'_2 \subseteq \frac 7 {10} Q'_3$ and $|Q'_1|\simeq |Q'_2| \simeq |Q'_3|$, and bump functions $\phi_{Q'_3,\s,3}$ adapted to $Q'_3$ uniformly in $\s \in \Z^2$ and supported in $\frac 9 {10} Q'_3$, such that $\phi_{Q'_3,\s,3}\equiv1$ on $\frac 9 {10} Q'_1+\frac 9 {10} Q'_2$.

Thus (\ref{12}) can be written as
 \begin{align}
 \sum_{\#\geq 1000}
2^{-(\#+1)\l}
\sum_{Q, {{Q'}}; \atop 2^\#{|Q|} \sim |Q'|
 } 
 \phi_{Q_1, \s,1}(\xi_1)
 &\phi_{Q_2, \s,2}(\xi_2)\phi_{Q_3,\s,3}(\xi_1+\xi_2)\cdot\nonumber\\
 &\widetilde{\phi}_{{Q'_1},\n,1} ({\xi_1+\xi_2}) \phi_{Q'_2, \n,2}(\xi_3)\phi_{Q'_3, \n,3}(\xi_1+\xi_2+\xi_3)\label{13}
\end{align}
where shifted dyadic quasi-cubes $Q$, $Q'$ in $\R^3$ are defined by $Q:=Q_1\times Q_2 \times Q_3, \;  Q':=Q'_1\times Q'_2 \times Q'_3$. Since any set of a shifted dyadic quasi-cubes in $\R^3$ can be split into $O(1)$ sparse subsets, we can assume that the sum (\ref{13}) runs over sparse collections of $Q$ and $Q'$ modulo finitely many such corresponding expressions. Then we can see that, for each shifted dyadic quasi-cubes $Q$ in such a sparse collection, there exists a unique shifted dyadic cube $\tilde{Q}$ in $\R^3$ such that $Q \subseteq \frac 7 {10} \tilde{Q}$ and $|Q| \sim |\tilde Q|$. Thus we can now assume that the sum (\ref{13}) runs
over sparse collections $\Q, \Q'$ of shifted dyadic cubes $Q, Q'$, respectively. Then, we can see that the multipliers of the type (\ref{13}) are well localized, which allows us to simplify the corresponding trilinear operator $T_{m_{\#, \l}}$ with the symbol $m_{\#,\l}$, where $m_{\#,\l}$ denotes the inner sum in (\ref{13}). More specifically, 
in order to establish $L^p$ estimates for the trilinear operator $T_{m_{\#, \l}}$, for each $\#\geq 1000$, we consider the quadrilinear form $\Lambda_{\#,\l}$ associated to $T_{m_{\#, \l}}$ defined by
\begin{align}
\int_{\scriptsize{\R}} T_{m_{\#, \l}}&(f_1, f_2, f_3)(x) f_4(x) dx
\nonumber\\
=&\int_{\scriptsize{\R}}
 \left(\int_{\scriptsize{\R^3}}
 m_{\#, \l}(\xi_1, \xi_2, \xi_3)
\widehat{f_1}(\xi_1)
\widehat{f_2}(\xi_2)
\widehat{f_3}(\xi_3)
e^{2 \pi i x (\xi_1+\xi_2+\xi_3)}
d\xi_1
d\xi_2
d\xi_3
\right)
f_4(x) 
dx
\nonumber\\
=&\int_{\scriptsize{\R^3}} 
 m_{\#, \l}(\xi_1, \xi_2, \xi_3)
\widehat{f_1}(\xi_1)
\widehat{f_2}(\xi_2)
\widehat{f_3}(\xi_3)
\widehat{f_4}(-\xi_1-\xi_2-\xi_3)
d\xi_1
d\xi_2
d\xi_3\nonumber
\end{align}
\begin{align}
=\sum_{Q\in\Q, {{Q'\in\Q'}}; \atop 2^\#{|Q|} \sim |Q'|
 } &
\int_{\xi_1+\xi_2+\xi_3+\xi_4=0} 
\phi_{Q_1, \s,1}(\xi_1)
\phi_{Q_2, \s,2}(\xi_2)
\phi_{Q_3,\s,3}(\xi_1+\xi_2)
\widetilde{\phi}_{{Q'_1}, \n, 1} ({\xi_1+\xi_2})
\cdot
\nonumber\\
&\phi_{Q'_2, \n,2}(\xi_3)\phi_{Q'_3,\n,3}(\xi_1+\xi_2+\xi_3)
\widehat{f_1}(\xi_1)
\widehat{f_2}(\xi_2)
\widehat{f_3}(\xi_3)
\widehat{f_4}(\xi_4)
d\xi_1
d\xi_2
d\xi_3
d\xi_4\nonumber
\end{align}
\begin{align}
=\sum_{Q\in\Q, {{Q'\in\Q'}}; \atop 2^\#{|Q|} \sim |Q'| } &
\int_{\xi_1+\xi_2+\xi_3+\xi_4=0} 
\widehat{f_1\ast \check \phi_{Q_1, \s,1}}(\xi_1)\;\;
\widehat{f_2\ast \check \phi_{Q_2, \s,2}}(\xi_2)\;\;
\widehat{\check \phi_{Q_3,\s,3} \ast \check{\widetilde {\phi}}_{{Q'_1}, \n,1}}(\xi_1+\xi_2)\nonumber\\
&\widehat{f_3\ast \check \phi_{Q'_2,\n,2}}(\xi_3)\;
\widehat{f_4\ast \check{\widetilde{\phi}}_{Q'_3,\n,3}}(\xi_4)\;
d\xi_1
d\xi_2
d\xi_3
d\xi_4\nonumber
\end{align}
\begin{align}
=\sum_{Q'\in\Q'}
\int_{\scriptsize{\R}}&
\left(
\left[
\sum_{Q\in\Q;\atop 2^\#{|Q|} \sim |Q'|}(
({f_1\ast \check \phi_{Q_1, \s,1}})(x)
({f_2\ast \check \phi_{Q_2, \s,2}})(x)
)
\ast
\check{ \phi}_{Q_3,\s,3} 
\right]
\ast \check{\widetilde {\phi}}_{{Q'_1}, \n, 1}\right)
(x)
\nonumber\\
&({f_3\ast \check \phi_{Q'_2, \n,2}})(x)
({f_4\ast \check{\widetilde{\phi}}_{Q'_3,\n,3}})(x)
dx\label{131}
\end{align}
by Plancherel's Theorem. Here, ${\widetilde{\phi}}_{Q'_3,\n,3}(\xi):={\phi}_{Q'_3,\n,3}(-\xi)$ and we redefine $Q'_3$ to be $-Q'_3$, which is the selected interval in $\R$ about the origin. Furthermore, (\ref{131}) is equal to
\begin{align}
\sum_{Q'\in\Q'} |Q'|^{\frac 3 2}
\int_{\scriptsize{\R}}&
\left\langle
\left[\sum_{Q\in\Q; 2^\#{|Q|} \sim |Q'|}
|Q|^{\frac 3 2}
\int_{\scriptsize{\R}}
\left\langle f_1, \Phi_{x'}^{{Q_1}, \s, 1} \right\rangle
\left\langle  f_2,\Phi_{x'}^{{Q_2}, \s, 2} \right\rangle
\overline\Phi_{x'}^{{Q_3}, \s, 3}
dx'
\right],\;
\Phi_{x}^{{Q'_1}, \n,1}
\right\rangle\nonumber\\
&\left\langle  f_3, \Phi_{x}^{{Q'_2},\n, 2} \right\rangle
\left\langle  f_4, \Phi_{x}^{{Q'_3}, \n, 3}\right \rangle
dx,\label{14}
\end{align}
where $\Phi_{x}^{Q_j, \s,j}(y)$ is defined by $|Q|^{-1/2} \overline{\check \phi_{Q_j,\s,j}(x-y)}$ and we define $\Phi_{x'}^{Q'_j, \n,j}(y) $ accordingly. Finally, by defining $\Phi_{P_j,t,\n, j}:=\Phi^{Q'_j,\n,j}_{ x_{\vec P}+|I_{\vec P}| t} $ and $\Phi_{Q_j,t',\s,j}:=\Phi^{Q_j,\s,j}_{x_{\vec Q}+|I_{\vec Q}| t'}$,
where $x_{\vec P}$ is the center of $I_{\vec P}$, we have that (\ref{14}) is equal to
$$\int_0^1 \int_0^1 \sum_{\vec P : Q_{\vec P} \in \Q'} \frac 1 {|I_{\vec P}|^{1/2}}
\left\langle B_{P_1,\s,t'}(f_1, f_2), \Phi_{P_1,t,\n,1} \right\rangle
\left\langle  f_3, \Phi_{P_2,t,\n,2} \right\rangle
\left\langle  f_4, \Phi_{P_3,t,\n,3}\right \rangle
dtdt',
$$
where 
$$ B_{P_1,\s,t'}(f_1, f_2)
=\sum_{\vec Q : Q_{\vec Q} \in \Q , \atop
 -\omega_{Q_3} \subset \omega_{P_1}, 
 2^\#{| \omega_{Q_3}|} \sim |\omega_{P_1}|}
 \frac 1 {|I_{\vec Q}|^{1/2}}
\left\langle f_1, \Phi_{Q_1,t',\s,1} \right\rangle
\left\langle  f_2, \Phi_{Q_2,t',\s,2} \right\rangle
\Bar \Phi_{Q_3, t',\s,3}
$$
and where $\vec P$ and $\vec Q$ range over all tri-tiles with frequency cube $Q_{\vec P}:=Q'$ and  $Q_{\vec Q}:=Q$, respectively. Observe that the collection of tri-tiles $\vec P$ and that of tri-tiles $\vec Q$ have rank 1 by construction. The condition $ -\omega_{Q_3} \subset \omega_{P_1}$ is automatic for nonzero summands. By redefining $Q_3$ to be $-Q_3$ and redefining $ \Phi_{Q_3, t',\s,3}$ accordingly, we can replace $\Bar \Phi_{Q_3, t',\s,3}$ by $ \Phi_{Q_3, t',\s,3}$ and the constraint $ -\omega_{Q_3} \subset \omega_{P_1}$ by $\omega_{Q_3} \subset \omega_{P_1}$ in the expression of $B_{P_1,\s,t'}$. Then Theorem 3.5 follows by integrating the conclusion of Theorem 4.9 over $t, t'$ because of the uniformity assumption of Theorem 4.9 (modulo $L^p$ estimates for the operator corresponding to
$m_{\I, R_M}$
 in (\ref{8}), which we show in Section 6). By the usual limiting argument, we can get rid of the finiteness condition on the collections $\vec \P$ and $\vec \Q$ of $\vec P$ and $\vec Q$, respectively. Finally, we have an operator, as those in Theorem 4.9, in terms of wave packets which are perfectly localized in frequency, but not in time space.

\vskip 1cm

\section{Fourier Coefficients}
In this section, we will prove the boundedness conditions of the Fourier coefficients $C_{\n}^{{{Q'}}}$ and $C_{\s}^{Q,\l}$ in $(\ref{fc1})$ and $(\ref{fc2})$, respectively,  which were necessary properties in the previous chapter. First, recall that
$$C_{\n}^{{{Q'}}} = \frac 1 {|{Q'}|^2} \int_{\scriptsize{\R^2}} m_2(\xi_2, \xi_3) \phi_{{Q'}}(\xi_2, \xi_3) e^{-2\pi i n_1 {\xi_2 \over {|{Q'}|}}} e^{-2\pi i n_2 {\xi_3 \over {|{Q'}|}}}d\xi_2 d\xi_3,$$  and 
$$C_{\s}^{Q,\l} = \frac 1 {{|{Q}|^2}} \int_{\scriptsize{\R^2}} m_1(\xi_1, \xi_2) \phi_Q(\xi_1, \xi_2)\left(\frac{\xi_2-\xi_1}{{|{Q}|}}\right)^\l \frac 1 {\l!}  e^{-2\pi i s_1 {\xi_1 \over {|{Q}|}}} e^{-2\pi i s_2 {\xi_2 \over {|{Q}|}}}d\xi_1 d\xi_2.$$
\begin{lemma}
$|C_{\n}^{{{Q'}}}| \lesssim C(\n)$, which is not depending on $Q'$, where $\n:=(n_1, n_2)$$\in\Z^2$. Similarly,
 $|C_{\s}^{Q,\l}| \lesssim C(\s)$, which is not depending on $Q$ and $\l$,
where $\s:=(s_1, s_2)$$\in\Z^2$. 
\end{lemma}

   \begin{proof}
 Let $|{Q}|=2^{k_1}$ and $|{{Q'}}|=2^{k_2}$ for $k_1, k_2 \in \Z$. We denote $m_2(\xi_2, \xi_3) \phi_{{Q'}}(\xi_2, \xi_3)$ by $m_{2, Q'}(\xi_2, \xi_3)$, which is the smooth restriction for the symbol $m_2$ to the cube $Q'$. Since $m_{2, Q'}$ is smooth, we can integrate by parts as much as we want. Before doing that, we have
 $$C_{\n}^{{{Q'}}} =\int_{\scriptsize{\R^2}} m_{2, Q'}(2^{k_2}\xi_2, 2^{k_2}\xi_3) e^{-2\pi i (n_1 \xi_2+ n_2 \xi_3)}d\xi_2 d\xi_3$$ 
 and because of the support $Q'$ of $m_{2,Q'}$, we can see that
 $|2^{k_2} \xi_2 -2^{k_2} \xi_3| \sim 2^{k_2}$ for any $(\xi_2, \xi_3) \in Q'$.
Then after taking integration by parts sufficiently many times, we have
$$ |C_{\n}^{{{Q'}}}| \lesssim  \frac {1}{(1+|n_1| )^M} \frac {1}{(1+|n_2|)^M} \leq \frac {1}{(1+|n_1| +|n_2|)^M}$$
for some large constant $M \in \Z$,
 because we have that $$|\partial^{\alpha}({m_{2, Q'}} (\xi, \eta))| \leq \frac {1}{\dist(\Gamma_2, (\xi, \eta))^{|\alpha|}}\sim \frac {1}{2^{{k_2}|\alpha|}},$$ 
 where $\Gamma_2=\{(\xi_2, \xi_3)\in \R^2 : \xi_2=\xi_3\}$
 so that $|\partial^{\alpha}({m_{2, Q'}} (2^{k_2}\xi_2, 2^{k_2}\xi_3))| \sim 1.$
We note that  $ \frac {1}{(1+|n_1| +|n_2|)^M}$ is a rapidly decreasing sequence depending only on $\n$, and not on $Q'$.

 Similarly, we can see that
$$ C_{\s}^{{{Q},\l}} =\int_{\scriptsize{\R^2}} m_1(2^{k_1}\xi_1, 2^{k_1}\xi_2)\phi_{Q}(2^{k_1}\xi_1, 2^{k_1}\xi_2)\left({\xi_2-\xi_1}\right)^\l \frac 1 {\l!}     e^{-2\pi i (s_1 \xi_1 + s_2 \xi_2 )}d\xi_1 d\xi_2.$$
Because of the support $Q$ of $\phi_{Q}$, we have
 $|2^{k_1}\xi_1 -2^{k_1}\xi_2| \sim 2^{k_1}.$
 Thus, we can easily see that 
 $$ C_{\s}^{{{Q},\l}} \lesssim  \frac {1} {(1+|s_1| +|s_2|)^{M'}}$$
 for some large constant $M' \in \Z$.
 We also note that  $ \frac {1} {(1+|s_1| +|s_2|)^{M'}}$
 is a rapidly decreasing sequence depending only on $\s$, and not on $Q, \l$.
 \end{proof}

 \vskip 1cm

 \section{Remainder Term}
 
 In this section, we will obtain $L^p$ estimates for the operator corresponding to $m_{\I,R_M}$ in (\ref{8}).
 Let us recall from (\ref{8}) that $m_{\I,R_M}$ be given by
 \begin{equation}\label{rmt}
  \sum_{Q, {{Q'}}; |Q|\ll|{Q'}| } m_1(\xi_1, \xi_2) \phi_Q(\xi_1, \xi_2) R_M(\xi_1, \xi_2)\phi_{Q'_2, \n,2}(\xi_3).
  \end{equation}
Here, the remainder term in the Taylor series in (\ref{remainder}) is defined by
$$R_M(\xi_1, \xi_2)= \frac 1 {M!} \phi_{Q'_1, \n,1}^{(M)}(\xi_{\theta})\left(\frac{\xi_2-\xi_1}{2}\right)^{M}$$
for a big number $M\in \Z$ and $\xi_{\theta} =(1-\frac \theta 2)\xi_2 +\frac \theta 2 \xi_1$ for some $\theta$ between $0$ and $1$.

 Let $|{Q}|=2^{k_1}$ and $|{{Q'}}|=2^{k_2}$ for $k_1, k_2 \in \Z$. Then, we have that (\ref{rmt}) is equal to 
 $$\sum_{\#\geq 1000}
 \sum_{Q, {{Q'}}\atop k_2-k_1=\# } m_1(\xi_1, \xi_2) \phi_Q(\xi_1, \xi_2)  \phi_{Q'_1, \n,1}^{(M)}(\xi_{\theta})\left(\frac{\xi_2-\xi_1}{2}\right)^{M} \frac 1 {M!} \phi_{Q'_2, \n,2}(\xi_3)$$
\begin{align}
=\sum_{\#\geq 1000}
\sum_{Q, {{Q'}}\atop  k_2-k_1+1=\# }
&
2^{(k_1-1) M}2^{-k_2 M}
\left( 
m_1(\xi_1, \xi_2) \phi_Q(\xi_1, \xi_2)
 \left(\frac{\xi_2-\xi_1}{2^{k_1}}\right)^{M} \frac 1 {M!} \right)\cdot \nonumber\\
& 
 2^{k_2 M}   \phi_{Q'_1, \n,1}^{(M)}(\xi_{\theta}) \phi_{Q'_2, \n,2}(\xi_3).
\label{15}
\end{align}
We compute a double Fourier series of $m_1(\xi_1, \xi_2) \phi_Q(\xi_1, \xi_2)
 \left(\frac{\xi_2-\xi_1}{2^{k_1}}\right)^{M} \frac 1 {M!}$ to see that (\ref{15}) is equal to
 \begin{equation}\label{296}
\sum_{\#\geq 1000}
\sum_{Q, {{Q'}}\atop k_2-k_1=\# }2^{-(\#+1) M}
\left( \sum_{\s \in  \scriptsize{\Z^2}} C_{\s}^{{Q}} \phi_{Q_1, \s,1}(\xi_1)\phi_{Q_2, \s,2}(\xi_2) \right)
 2^{k_2 M}   \phi_{Q'_1, \n,1}^{(M)}(\xi_{\theta}) \phi_{Q'_2, \n,2}(\xi_3),
\end{equation}
where $C_{\s}^{{Q}}=C_{\s}^{Q,\l}$ in (\ref{fc2}) when $\l=M$. Thus, we can obtain that $|C_{\s}^{Q,M}| \lesssim C(\s)$ by the same proof in Section 5. Furthermore,
we observe that if a dyadic interval ${Q'_1}$ is given by ${2^{k_2}}[m, m+1]$, for some $m \in \Z$, then $2^{k_2 M}   \phi_{Q'_1, \n,1}^{(M)}(\xi_{\theta})$ can be denoted by
 $\widetilde{\phi}^{(M)}_{n_1}(\frac {\xi_{\theta}}{2^{k_2}}-m)$, where $\widetilde{\phi}_{n_1}(\eta):=\phi(\eta)e^{2\pi i n_1 \eta}$ for a bump function $\phi$ adapted to $[0, 1]$. Therefore, (\ref{296}) can be majorized by
\begin{equation}\label{16}
\sum_{\s \in \scriptsize{\Z^2}} C(\s)
\sum_{\#\geq 1000}2^{-(\#+1) M}
\sum_{Q, {{Q'}}\atop k_2-k_1=\# }
 \phi_{Q_1, \s,1}(\xi_1)\phi_{Q_2, \s,2}(\xi_2) 
 \widetilde{\phi}^{(M)}_{n_1}\left(\frac {\xi_{\theta}}{2^{k_2}}-m\right)
  \phi_{Q'_2, \n,2}(\xi_3).
\end{equation}
 
 For simplicity, we denote $$\widetilde{m}_\#:=\sum_{Q, {{Q'}}\atop k_2-k_1=\# } \phi_{Q_1, \s,1}(\xi_1)\phi_{Q_2, \s,2}(\xi_2) 
  \widetilde{\phi}^{(M)}_{n_1}\left(\frac {\xi_{\theta}}{2^{k_2}}-m\right)
  \phi_{Q'_2, \n,2}(\xi_3).$$
Then, it is easy to remark that the operator corresponding to $m_{\I,R_M}$ can be majorized by $\sum_{\s \in  \scriptsize{\Z^2}} C(\s) \sum_{\#\geq 1000} 2^{-(\#+1) M} T_\#$, where $T_\#$ has a symbol $\widetilde{m}_\#$. Now, we claim that the symbol $\widetilde{m}_\#$ satisfies
\begin{equation}\label{RRM}
|\partial^{\alpha}(\widetilde{m}_\#(\xi))|\lesssim 
2^{\#|\alpha|}
 \frac{1}{\dist(\Gamma,  \xi)^{|\alpha|}}
\end{equation}
 for many multi-indices $\alpha$, where $\Gamma = \{(\xi_1, \xi_2, \xi_3)\in R^3: \xi_1=\xi_2=\xi_3\}$. If this claim holds, then the Coifman-Meyer Theorem
implies that $T_{\#}$ is bounded with a bound that is of type $O(2^{100\#})$, say.
This completes the proof for the desired $L^p$ estimates for the operator corresponding to $m_{\I,R_M}$ because of the big decaying factor $2^{-(\#+1) M}$ in (\ref{16}). These
arguments show that after justifying the claim (\ref{RRM}), Theorem 4.9 is the only one that remains
to be proved. 

In order to prove the claim (\ref{RRM}), fix $\xi_o \in \R^3$. Then, there exists a unique shifted dyadic quasi-cubes $Q_1 \times Q_2 \times Q_2'$ containing $\xi_o$ in $\R^3$ so that we obtain 
$$|
\partial^{\alpha}(\widetilde{m}_\#(\xi_o))|
 \leq
 \left| \partial^{\alpha}   \left[ \phi_{Q_1, \s,1}(\xi_1) \phi_{Q_2, \s,2}(\xi_2) 
  \widetilde{\phi}^{(M)}_{n_1}\left(\frac {\xi_{\theta}}{2^{k_2}}-n\right)
   \phi_{Q'_2, \n,2}(\xi_3)
\right]\bigg|_{\xi =\xi_o}
 \right|$$
 because of localization of $\phi_Q$ functions. Then, it is easy to check that 
 $$|
\partial^{\alpha}(\widetilde{m}_\#(\xi_o))|
\lesssim 2^{\#|\alpha|} \cdot 2^{-k_2 |\alpha|}
\lesssim
2^{\#|\alpha|}
 \frac{1}{\dist(\Gamma,  \xi_o)^{|\alpha|}}$$
 as $\dist(\Gamma,  \xi_o) \lesssim 2^{k_2}$ where $\xi_o \in Q_1 \times Q_2 \times Q_2'$. This completes the proof of the claim (\ref{RRM}).

\vskip 1cm

\section{Tile Norms and Known Estimates}
In order to establish the estimates (\ref{first}) for the forms $\Lambda^{\#}_{\vec\P, \vec\Q}$,
the standard approach is to organize our collections of tri-tiles $\vec\P,\vec\Q$ into trees as in \cite{GL} and consider the standard tile norms. We first review standard definitions and comments for trees from \cite{MTT2}. We will henceforth assume that $\vec\P$ and $\vec\Q$ are sparse of rank 1.

\begin{definition} For any $i=1,2,3$ and a tri-tile $\vec P_T \in \vec\P$, define a i-tree with top $\vec P_T$ to be a collection $T \subseteq \vec \P$ of tri-tiles such that
$$P_i \leq P_{T,i} \:for\: all\: \vec P \in T,$$
where $P_{T,i}$ is the $i$ component of $\vec P_T$ . We write $I_T$ and $\omega_{T,i}$ for $I_{\vec P_T}$ and $\omega_{P_{T},i}$ respectively.
We say that $T$ is a tree if it is a i-tree for some  $i=1,2,3$.
\end{definition}

Observe that a tree $T$ does not necessarily have to contain its top $\vec P_T$ .

\begin{definition} Fix $1 \leq i \leq 3$. Two trees $T, T'$ are said to be strongly $i$-disjoint if
\begin{enumerate}
\item $P_i \neq P'_i$ for all $\vec P \in T, \vec P' \in T'$.
\item For any $\vec P \in T, \vec P' \in T'$ with $2\omega_{P_i} \cap 2\omega_{P'_
i}\neq \emptyset$, one obtains $I_{\vec P' } \cap I_T = \emptyset$ and $I_{\vec P} \cap I_{T'} = \emptyset$.
\end{enumerate}
\end{definition}
Note that if $T$ and $T'$ are strongly $i$-disjoint, then we can see that two tiles $I_P \times 2\omega_{P_i}$ and $I_{P'} \times 2\omega_{P'_i}$ are disjoint for all $\vec P \in T, \vec P' \in T'$. Given that $\vec\P$ is sparse, we can see that if $T$ is an $i$- tree, then we have either
$\omega_{P_j} = \omega_{P'_j}$ or
$2\omega_{P_j} \cap 2\omega_{P'_j}= \emptyset$ for all $\vec P , \vec P' \in T$ and $j \neq i$. 
\vskip.5cm

We now recall the standard tile norms from \cite{MTT1} and \cite{MTT2}.
In the remainder of this paper we shall estimate expressions of the form
\begin{equation}\label{discrete}
\left| \sum_{\vec P \in \vec\P} \frac 1 {|I_{\vec P}|^{1/2}} a_{P_1}^{(1)}a_{P_2}^{(2)}a_{P_3}^{(3)}\right|,
\end{equation}
where $\vec \P$ is a collection of tri-tiles and for $i= 1, 2, 3$, $a_{P_i}^{(i)}$ are complex numbers for $P_i$, where $\vec P=(P_1, P_2, P_3) \in \vec \P.$

In some cases such as the associated expression for the Bilinear Hilbert transform, we just have
\begin{equation}\label{17}
a_{P_i}^{(i)} = \langle f_i, \Phi_{P_i} \rangle,
\end{equation}
for $i=1,2,3$,
however, here we will have more complicate sequences $a_{P_3}^{(3)}$ when dealing with $\Lambda^{\#}_{\vec \P,\vec \Q}$.

In \cite{MTT1} the following standard tile norms were introduced:

\begin{definition}Suppose that $\vec \P$ is a finite collection of tri-tiles. For $i= 1, 2, 3$, suppose that $(a_{P_i}^{(i)})_{ \vec P\in\vec \P}$ is a sequence of complex numbers. We define the size of the sequence by
$$\size_i((a_{P_i}^{(i)} )_{ \vec P\in\vec \P} ) := \sup_{T\subset \vec\P}( \frac 1 {|I_T|} 
\sum_{\vec P\in T}
|a_{P_i}^{(i)}|^2)^{1/2},$$
where the supremum ranges over all j-trees $T$ in $\vec \P$ for some $j \neq i$. 

We also define the energy of the sequence by
$${\energy}_i((a_{P_i}^{(i)})_{\vec P\in\vec \P}) := \sup_{n\in  \scriptsize{\Z}}\sup_{\T}2^n(\sum_{T\in\T}
|I_T |)^{1/2},$$ 
where the inner supremum ranges over all collections $\T$ of strongly $i$-disjoint trees in $\vec \P$ such that
$$(\sum_{\vec P\in T} |a_{P_i}^{(i)} |^2)^{1/2} \geq 2^n |I_T |^{1/2}$$
for all $T \in \T$, and
$$(\sum_{\vec P\in T'} |a_{P_i}^{(i)} |^2)^{1/2} \leq 2^{n+1} |I_{T'} |^{1/2}$$
for all sub-trees $T' \subset T \in \T$.

\end{definition}

Observe that the number $a_{P_i}^{(i)}$ in Definition 7.3 is only associated with the i-tile $P_i$ rather than the full tri-tile $\vec P$. We can understand the size of a sequence as a measure the extent to which the sequence can concentrate on a single
tree. 

Since the size can be thought of as a phase-space variant of the BMO norm, we obtain the following relevant variant of the John-Nirenberg inequality for the size:

\begin{lemma}
Suppose that $\vec \P$ is a finite collection of tri-tiles and that for $i = 1, 2, 3$, $(a_{P_i}^{(i)} )_{ \vec P\in \vec\P}$ is a
sequence of complex numbers. Then
$$\size_i((a_{P_i}^{(i)} )_{ \vec P\in\vec \P} )\sim \sup_{T \subset \vec\P}  \frac 1 {|I_T|} \| 
(\sum_{\vec P\in T}
|a_{P_i}^{(i)}|^2 \frac {\chi_{I_{\vec P}}} {|I_{\vec P}|}
)^{1/2}\|_{L^{1, \infty}(I_T)}
$$
where the supremum ranges over all j-trees $T\subset \vec \P$ for some $j \neq i$.
\end{lemma}

\begin{proof}
See Lemma 4.2 in \cite{MTT1}.
\end{proof}

By duality argument we can see that the following lemma:
\begin{lemma} 
Suppose that $\vec \P$ is a finite collection of tri-tiles and that $(a_{P_i}^{(i)} )_{ \vec P\in \vec\P}$, for $i = 1, 2, 3$, is a
sequence of complex numbers. 
Then, there
exists a collection $\T$ of strongly i-disjoint trees, and complex coefficients $c_{P_i}$, for all $\vec P \in \bigcup_{T \in \T} T$, such that
     $${\energy}_i((a_{P_i}^{(i)})_{ \vec P\in\vec \P}) \sim | \sum_{T \in \T} \sum_{\vec P \in T} a_{P_i}^{(i)} \bar{c}_{P_i} |,$$
and such that
$$\sum_{\vec P \in T'} |c_{P_i}|^2 \lesssim \frac{|I_{T'}|} {\sum_{T \in \T} |I_T|}$$
for all $T \in \T$ and all sub-trees $T' \subseteq T$ of $T$.
\end{lemma}
\begin{proof}
The proof follows by setting that $n$, $\T$ are the extremizers in the definition of the energy and that 
$c_{P_i}:= 2^n({\sum_{T \in \T} |I_T|})^{-1/2} a_{P_i}^{(i)}$ for all $\vec P \in \bigcup_{T \in \T} T$. 
\end{proof}



To establish the estimates on (\ref{discrete}), we recall the following standard combinatorial tool from \cite{MTT1}:
\begin{proposition}
Suppose that $\vec \P$ is a finite collection of tri-tiles and that $(a_{P_i}^{(i)} )_{ \vec P\in \vec\P}$, for $i = 1, 2, 3$, is a
sequence of complex numbers. 
 Then
$$| \sum_{ \vec P\in\vec \P}  \frac 1 {|I_{\vec P}|^{1/2}}
 a^{(1)}_{P_1}  a^{(2)}_{P_2} a^{(3)}_{P_3}|
  \lesssim
   \prod_{i=1}^3 
\size_i((a^{(i)}_{P_i} )_{ \vec P\in \vec \P} )^{\theta_i}{\energy}_i((a^{(i)}
_{P_i} )_{ \vec P\in \vec \P} )^{1-\theta_i} $$
for any $0 \leq \theta_1, \theta_2, \theta_3< 1$ with $\theta_1+\theta_2+\theta_3=1$, with the implicit constant depending on the $\theta_i$.
\end{proposition}
\begin{proof}
See the Appendix \cite{MTT2}.
\end{proof}

 In order to use Proposition 7.6, we will need some estimates on size and
energy. We have the following lemmas from \cite{MTT2} for $a^{(i)}_{P_i}$, $i=1,2,3,$ defined in (\ref{17}):
\begin{lemma} Suppose that $\vec \P$ is a finite collection
of tri-tiles. Suppose that $E_i$, for $i= 1, 2, 3$, is a subset of $\R$ with finite
measure and $f_i$ is a function in $X(E_i)$. Then, we have
$$\size_i(( \langle f_i , \Phi_{P_i} \rangle)_{ \vec P\in \vec \P}) \lesssim \sup_{ \vec P\in \vec \P}
\int_{E_i} \frac {\tilde\chi_{I_{\vec P}}^M} {|I_{\vec P}|} $$
for all $M$, where the implicit constant depends on $M$.
\end{lemma}
\begin{proof}
See Lemma 6.8 in \cite{MTT2}.
\end{proof}

\begin{lemma} Suppose that $\vec \P$ is a finite collection
of tri-tiles. Suppose that $E_i$, for $i= 1, 2, 3$, is a subset of $\R$ with finite
measure and $f_i$ is a function in $X(E_i)$. Then, we have
$$\energy_i(( \langle f_i, \Phi_{P_i} \rangle)_{ \vec P\in\vec \P} ) \leq \|f_i\|_2.$$

\end{lemma}
\begin{proof}
See Lemma 6.7 in \cite{MTT2}.
\end{proof}

\vskip 1 cm

\section{Estimates of Size and Energy of $a^{(3), \#}_{Q_3}$}

 In this section, we will provide the relevant size and energy estimates for the form $\Lambda^{\#}_{\vec \P, \vec \Q}$ in Theorem 4.9. In order to obtain more complicated size and energy estimates than those of $a_{P_i}^{(i)} = \langle f_i, \Phi_{P_i} \rangle$ described in Lemma 7.7 and Lemma 7.8, it shall be more convenient to rewrite $\Lambda^{\#}_{\vec \P, \vec \Q}$ as
$$\Lambda^{\#}_{\vec \P, \vec \Q}(f_1, f_2, f_3, f_4) =  \sum_{\vec Q \in \vec\Q} \frac {1}{|I_{\vec Q}|^{1/2}}
a^{(1)}_{Q_1}a^{(2)}_{Q_2}a^{(3),\#}_{Q_3},$$
 where
 \begin{align}
 a^{(1)}_{Q_1}&:= \langle f_1 , \Phi_{Q_1}\rangle \nonumber\\
a^{(2)}_{Q_2}&:=\langle f_2 , \Phi_{Q_2}\rangle  \nonumber\\
a^{(3),\#}_{Q_3}&:=  \sum_{\vec P \in \vec\P:\omega_{Q_3}\subseteq \omega_{P_1} \atop 2^{\#}|\omega_{Q_3}| \sim |\omega_{P_1}|}
\frac {1} {|I_{\vec P}|^{1/2}}
\langle f_3 , \Phi_{P_2}\rangle 
   \langle f_4 , \Phi_{P_3}\rangle
  \langle  \Phi_{P_1},  \Phi_{Q_3}\rangle\label{tiles}
\end{align}
by reversing the order of summation. As a consequence, the inner summation has a stronger spatial localized than the outer summation because the $\vec P$ tiles have a narrower spatial interval than the $\vec Q$ tiles by the constraint $\omega_{Q_3}\subseteq \omega_{P_1}$.  

Compared to the form $\Lambda_{ \vec\P,\vec\Q}$ in Theorem 4.8 of \cite{MTT2}, we have the form $\Lambda^{\#}_{ \vec\P,\vec\Q}$ with extra constraint $2^{\#}|\omega_{Q_3}| \sim |\omega_{P_1}|$. Because of this constraint, we are no longer able to apply one of the essential techniques to estimate the tile norms of the type of $a^{(3)}_{Q_3}$ (described in Lemma 8.2 of \cite{MTT2} and Lemma 6.1 of \cite{MTT1}). This technique exploits the symmetry of  $\vec P$ and $\vec Q$ to decouple them. Hence we have 
$$a^{(3)}_{Q_3}:=  \sum_{\vec P \in \vec\P:\omega_{Q_3}\subseteq \omega_{P_1}}
\frac {1} {|I_{\vec P}|^{1/2}}
\langle f_3 , \Phi_{P_2}\rangle 
   \langle f_4 , \Phi_{P_3}\rangle
  \langle  \Phi_{P_1},  \Phi_{Q_3}\rangle$$
  $$= \sum_{\vec P \in \vec\P'}
\frac {1} {|I_{\vec P}|^{1/2}}
\langle f_3 , \Phi_{P_2}\rangle 
   \langle f_4 , \Phi_{P_3}\rangle
  \langle  \Phi_{P_1},  \Phi_{Q_3}\rangle,$$
  where $\vec\P'$ is the sub-collection of $\vec\P$, which only depends on a fixed $i$-tree $T \subseteq \vec\Q$, for some $i=1,2$, and not on each $\vec Q\in \vec \Q$.
Thus, here, we will need a new trick to deal with the tile norms of the type of $a^{(3),\#}_{Q_3}$ by takeing advantage of the extra constraint $2^{\#}|\omega_{Q_3}| \sim |\omega_{P_1}|$.

 We first set out some general definitions, which shall be useful in this chapter.
 Define that a collection $\{\omega\}$ of intervals is $\emph{lacunary\; around \;the\; frequency\; \xi}$ if we have $\dist(\xi, \omega) \sim |\omega|$ for all $\omega$ in the collection.
We also define a {\em modulated Calder\'{o}n-Zygmund  operator} to be any operator $T_K$ of the form
    $${T_K}f(x)=\int K(x, y)f(y) dy,$$
  where $x, y \in \R$, and the (possibly vector-valued) kernel $K$ satisfies the estimates
  $$|K(x, y)| \lesssim 1/|x-y|$$
  and 
  $$|\nabla_{x,y} (e^{2\pi i (x\xi + y\eta)} K(x, y))| \lesssim 1/|x-y|^2$$
  for all $x\neq y$ and for some $\xi, \eta \in \R$. Observe that a modulated Calder\'{o}n-Zygmund operator is the composition of an ordinary Calder\'{o}n-Zygmund operator with modulation operators such as $f \mapsto e^{2\pi i \xi} \cdot f$. Thus we see that $T_K$ is bounded on $L^p$, for all $1<p<\infty$, and is also weak-type $(1, 1)$ by standard Calder\'{o}n-Zygmund theory.
  
 In order to estimate the size of $a^{(3),\#}_{Q_3}$, we need the following lemma:
 \begin{lemma}
  Let $\vec \P$ and $\vec \Q$ be finite collections of tri-tiles and $T \subset \vec \Q$ be a $i$-tree for some $i=1,2$. Define the collections of intervals, $\{\omega_{Q_3}\}_T$ and $\{\omega_{P_i}\}_{ \vec \P, T}$, for $i=1,2,3$, by $$ \{\omega_{Q_3}\}_{\vec Q \in T} := \{\omega_{Q_3} : \vec Q \in T\} $$
   $${ \rm{and}}\; \;\;\; \; \;\;\;\;\; \;\; \{\omega_{P_i}\}_{ \vec \P, T}:=\{\omega_{P_i} : \vec P \in \vec \P, \omega_{Q_3} \subset \omega_{P_1}, 2^{\#} |\omega_{Q_3}| \sim |\omega_{P_1}| \;  { \rm{where}}\;  \vec Q \in T\}.$$
Then, the collections $\{\omega_{Q_3}\}_{\vec Q \in T}$, $\{\omega_{P_2}\}_{ \vec \P, T}$, and $\{\omega_{P_3}\}_{ \vec \P, T}$ are lacunary, however the collection $\{\omega_{P_1}\}_{ \vec \P, T}$ is non-lacunary.
  \end{lemma}
  
  \begin{proof}
  Without loss of generality, we shall assume that $T$ is 1-tree in this proof. 
  
  First, to show that the collection $\{\omega_{P_1}\}_{ \vec \P, T}$
  is non-lacunary, pick any $\xi_o \in \omega_{T,1}$. It is easy to see that $\xi_o \in \omega_{T,1} \subset 3\omega_{Q_1}$ for all $\vec Q \in T$, henceforth there exists a fixed number $C_o$ such that
  $\xi_o \in C_o \omega_{Q_3}$
  for all $\vec Q \in T$. Therefore, we conclude that $ \xi_o \in C_0 \omega_{P_1}  $  for all $\omega_{P_1} \in  \{\omega_{P_1}\}_{ \vec \P, T}$, and hence the collection $\{\omega_{P_1}\}_{ \vec \P, T}$
  is non-lacunary.
  
  Then, it follows that the collection $\{\omega_{P_i}\}_{ \vec \P, T}$, for each $i=2,3$, is lacunary with the same $\xi_0$, because we know that, for each $\vec P \in \vec \P$, $\omega_{P_2}$ and $ \omega_{P_3}$ are a few steps away from $ \omega_{P_1}.$ 
  
  Now, we prove that $\{ \omega_{Q_3}\}_{\vec Q \in T}$ is lacunary:
For any $\vec Q \in T$, we have $Q_1 \leq Q_{T, 1}$. Then, because we have rank 1, we obtain $Q_2 \lesssim Q_{T, 2}$ and $Q_3 \lesssim Q_{T, 3}$. In particular, if we furthermore assume $Q_1 < Q_{T,1}$, then we have $|10^9 \omega_{T,1}|<|\omega_{Q_1}| $
by sparseness.
Thus, we obtain
$Q_3 \lesssim Q_{T,3} \;  {\rm{and}} \; Q_3 \nleq Q_{T,3}.$
Thence, we have
$$10^7 \omega_{T,3} \subset 10^7\omega_{Q_3} \; {\rm but} \; 3\omega_{T,3} \nsubseteq 3\omega_{Q_3}.$$
Again, pick any $\xi_o \in \omega_{T,1}$. Then, we conclude that $\dist(\xi_o , \omega_{Q_3}) \approx C_o |\omega_{Q_3}|.$
  \end{proof}
  Now we can obtain the size estimates for $a^{(3),\#}_{Q_3}$.
 \begin{lemma} 
For $ j = 3,4$, let $E_j$ be sets of finite measure in $\R$ and $f_j$ be functions in $X(E_j)$.
Then, we have
\begin{equation}\label{19}
\size_3((a^{(3),\#}_{Q_3})_{ \vec Q\in \vec \Q}) \lesssim \sup_{ \vec Q\in \vec \Q}
\left( \int_{E_3} \frac {\tilde\chi_{I_{\vec Q}}^M} {|I_{\vec Q}|} \right)^{1-\theta}
\left( \int_{E_4} \frac {\tilde\chi_{I_{\vec Q}}^M} {|I_{\vec Q}|} \right)^{\theta}
\end{equation}
for any $0 < \theta < 1$ and $M > 0$, with the implicit constant depending on $\theta, M$.
\end{lemma} 
\begin{proof}
We first recall that $a^{(3),\#}_{Q_3}$ has been defined by
$$a^{(3),\#}_{Q_3}=  \sum_{\vec P \in \vec\P:\omega_{Q_3}\subseteq \omega_{P_1} \atop 2^{\#}|\omega_{Q_3}| \sim |\omega_{P_1}|}
\frac {1} {|I_{\vec P}|^{1/2}}
\langle f_3 , \Phi_{P_2}\rangle 
   \langle f_4 , \Phi_{P_3}\rangle
  \langle  \Phi_{P_1},  \Phi_{Q_3}\rangle.$$
By Lemma 7.4, it suffices to show that, for any $1$ or $2$-tree $T\subset \vec Q$, we have
$$ \frac 1 {|I_T|
}
 \left\| 
\left(\sum_{\vec Q\in T}
|a_{Q_3}^{(3),\#}|^2 \frac {\chi_{I_{\vec Q}}} {|I_{\vec Q}|}
\right)^{1/2}
\right\|_{L^{1, \infty}(I_T)} 
\lesssim 
\sup_{ \vec Q\in \vec \Q}
\frac 1 {|I_{\vec Q}|
}
\left( \int_{E_3}{\tilde\chi_{I_{\vec Q}}^M} \right)^{1-\theta}
\left( \int_{E_4} {\tilde\chi_{I_{\vec Q}}^M}  \right)^{\theta}
.$$
 We may assume that $T$
contains its top $P_T$, since in the other case we could decompose $T$ into disjoint trees containing its top.
Henceforth, it suffices to estimate
\begin{equation}\label{20}
 \left\| 
\left(\sum_{\vec Q\in T}
|a_{Q_3}^{(3),\#}|^2 \frac {\chi_{I_{\vec Q}}} {|I_{\vec Q}|}
\right)^{1/2}
\right\|_{L^{1, \infty}(I_T)} 
\lesssim 
  \left( \int_{E_3}\tilde\chi_{I_T}^M\right)^{1-\theta}
  \left( \int_{E_4}\tilde\chi_{I_T}^M\right)^{\theta}
\end{equation}
for any $1$ or $2$-tree $T\subset \vec Q$.
Now we fix $T$. To prove (\ref{20}), we first consider the case when both $f_3$ and $f_4$ are supported on $5I_T.$
We may then assume that $E_3,E_4 \subset 5I_T$. By Plancherel's Theorem we see that
 \begin{equation}\label{22}
 \langle  \Phi_{P_1},  \Phi_{Q_3}\rangle
 =\langle \widehat{\Phi}_{P_1}, \widehat{\Phi}_{Q_3}\rangle.
 \end{equation}
 Now we pick a Schwartz function $\psi_{\omega_{Q_3}} $ so that
  $\supp\; \widehat{\psi}_{\omega_{Q_3}} \subseteq  \omega_{Q_3}$ and
  $\widehat{\psi}_{\omega_{Q_3}} \equiv 1$ on $\frac {9}{10} \omega_{Q_3}.$
   Then, (\ref{22}) can be rewritten as
 \beq
  \langle \widehat{\Phi}_{P_1}, \widehat{\Phi}_{Q_3} \cdot  \widehat{\psi}_{\omega_{Q_3}} \rangle
   &=& \langle \widehat{\Phi_{P_1} \ast {\psi}_{\omega_{Q_3}}} , \widehat{\Phi}_{Q_3} \rangle\\
  &=&2^{-\frac {\#} {2}} \langle 2^{{\frac {\#} {2}}} \widehat{\Phi_{P_1} \ast {\psi}_{\omega_{Q_3}} }, \widehat{\Phi}_{Q_3} \rangle\\
  &:=& 2^{-\frac {\#} {2}} \langle \Tilde{\Phi}_{\tilde{P_1}} , {\Phi}_{Q_3} \rangle,
  \eeq
  where 
  $ \Tilde{\Phi}_{\tilde{P_1}} :=
  2^{{\frac {\#} {2}}} {\Phi_{P_1} \ast {\psi}}_{\omega_{Q_3}}.$ We note that $\Tilde{\Phi}_{\tilde{P_1}}$ is an $L^2$-normalized bump adapted to $I_{\vec P,{2^{\#}}}$, where $I_{\vec P,{2^{\#}}}$ denotes the unique dyadic interval of length
  $2^{\#}|I_{\vec P}|$ containing $I_{\vec P}$. 
  Thus, we can see that
  $$a^{(3),\#}_{Q_3}=
  2^{-\frac {\#} {2}}
   \sum_{\vec P \in \vec\P:\omega_{Q_3}\subseteq \omega_{P_1} \atop 2^{\#}|\omega_{Q_3}| \sim |\omega_{P_1}|}
\frac {1} {|I_{\vec P}|^{1/2}}
\langle f_3 , \Phi_{P_2}\rangle 
   \langle f_4 , \Phi_{P_3}\rangle
 \langle \Tilde{\Phi}_{\tilde{P_1}} , {\Phi}_{Q_3} \rangle$$
 Now we claim that
  $$a^{(3),\#}_{Q_3}=
  2^{-\frac {\#} {2}}
   \sum_{\vec P' \in \vec\P'}
\frac {1} {|I_{\vec P}|^{1/2}}
\langle f_3 , \Phi_{P_2}\rangle 
   \langle f_4 , \Phi_{P_3}\rangle
 \langle \Tilde{\Phi}_{\tilde{P_1}} , {\Phi}_{Q_3} \rangle,$$
where the collection $\vec\P'$ of tri-tiles is defined by
 $$\vec \P' (T) = \{\vec P'=\tilde{P_1} \times P_2 \times {P_3} : \vec P =P_1 \times P_2 \times P_3  \in \vec \P \;{\rm and} \;  \tilde{P_1}=I_{\vec P,{2^{\#}}} \times  \omega_{Q_3} $$
 $$\;{\rm for}\; {\rm some}\; \vec Q \in T\; {\rm with} \;
 \omega_{Q_3}\subseteq \omega_{P_1} \;{\rm and }\; 2^{\#}|\omega_{Q_3}| \sim |\omega_{P_1}|\}.$$

 We claim that the sub-collection $\vec \P' \subset \vec \P$ depends on $\vec \P$ and $T \subset \vec \Q$, and not on each $\vec Q \in T$. We will prove this claim in Lemma 8.3.
Using the claim, we can see that 
  $$a^{(3),\#}_{Q_3}=2^{-\frac {\#} {2}}
  \left\langle B_{\vec \P'}(f_3, f_4), \Phi_{Q_3} \right\rangle,
$$
where $$ B_{\vec \P'}:= \sum_{\vec P' \in \vec\P'}
\frac {1} {|I_{\vec P}|^{1/2}}
\langle f_3 , \Phi_{P_2}\rangle 
   \langle f_4 , \Phi_{P_3}\rangle
 \Tilde{\Phi}_{\tilde{P_1}}.$$
Thus, in order to prove (\ref{20}), it suffices to show that
$$\left\|
\left(        
 \sum_{\vec Q\in T}
 |2^{-\frac {\#} {2}}
  \left\langle B_{\vec \P'}(f_3, f_4), \Phi_{Q_3} \right\rangle |^2
   \frac {\tilde \chi^{100}_{I_{\vec Q}}} {|I_{\vec Q}|}
\right)^{1/2}
\right\|_{L^{1, \infty}} 
\lesssim 
|E_3|^{1-\theta}
|E_4|^\theta.$$
The vector-valued operator
$$F \longmapsto
\left(\langle F, \Phi_{Q_3} \rangle 
 \frac {\tilde \chi^{50}_{I_{\vec Q}}} {{|I_{\vec Q}|}^{1/2}}
\right)_{\vec Q \in T}
$$
is a modulated Calder\'{o}n-Zygmund operator since we know that the collection $\{\omega_{Q_3}\}_{\vec Q \in T}$ of intervals is lacunary, so it suffices to show that
$$2^{-\frac {\#} {2}}
\|B_{\vec \P'}(f_3, f_4)
\|_1
\lesssim
\|f_3\|_{1/(1-\theta)}
\|f_4\|_{1/{\theta}}.
$$
We observe that 
\beq
2^{-\frac {\#} {2}}
\|B_{\vec \P'}(f_3, f_4)
\|_1 
&=& 
2^{-\frac {\#} {2}}
\sum_{\vec P' \in \vec\P'}
\frac {1} {|I_{\vec P}|^{1/2}}
|\langle f_3 , \Phi_{P_2}\rangle |
 |  \langle f_4 , \Phi_{P_3}\rangle|
\int_{\scriptsize{\R}} | \Tilde{\Phi}_{\tilde{P_1}}| dx \\
&=&
\sum_{\vec P' \in \vec\P'}
|\langle f_3 , \Phi_{P_2}\rangle |
 |  \langle f_4 , \Phi_{P_3}\rangle|\label{ppp}
  \eeq
 as $\frac  {\Tilde{\Phi}_{\tilde{P_1}}(x)} {2^{\frac {\#} {2}}  {|I_{\vec P}|^{1/2}}}$ is $L^1$-normalized. Then, it is equal to 
 \begin{align}
&\sum_{\vec P' \in \vec\P'}
\frac{|\langle f_3 , \Phi_{P_2}\rangle |}{|I_{\vec P}|^{1/2}}
\cdot
 \frac {|  \langle f_4 , \Phi_{P_3}\rangle|}{|I_{\vec P}|^{1/2}}
\cdot
|I_{\vec P}|\nonumber\\
&\lesssim
\int_{\scriptsize{\R}} \sum_{\vec P' \in \vec\P'}
\frac{|\langle f_3 , \Phi_{P_2}\rangle |}{|I_{\vec P}|^{1/2}}
\cdot
 \frac {|  \langle f_4 , \Phi_{P_3}\rangle|}{|I_{\vec P}|^{1/2}}
\cdot
 {\tilde \chi^{100M}_{I_{\vec P}}}(x) dx\nonumber\\
 &\lesssim
 \int_{\scriptsize{\R}} \left(  \sum_{\vec P' \in \vec\P'}
 \frac{|\langle f_3 , \Phi_{P_2}\rangle |^2}{|I_{\vec P}|} \cdot 
 {\tilde \chi^{100M}_{I_{\vec P}}}
  \right)^{1/2}
 \left(  \sum_{\vec P' \in \vec\P'}
 \frac{|\langle f_4 , \Phi_{P_3}\rangle |^2}{|I_{\vec P}|} \cdot 
 {\tilde \chi^{100M}_{I_{\vec P}}}
  \right)^{1/2}
  dx \nonumber\\
  &\lesssim
 \int_{\scriptsize{\R}} |S_2(f_3)(x)| \;|S_3(f_4)(x)| dx, \label{23}
\end{align}
where  $S_jf:=\left( \langle f, \Phi_{P_j} \rangle
\frac {\tilde \chi^{50M}_{I_{\vec P}}}{|I_{\vec P}|^{1/2}}
\right)_{\vec P' \in \vec \P'\; }$.
Then, since we know that the collections of intervals, $\{ \omega_{P_2}\}_{ \vec \P, T}$ and $\{ \omega_{P_3}\}_{ \vec \P, T}$, are lacunary, we can majorize (\ref{23}) by
$$ \|S_2(f_3)(x)\|_{1/(1-\theta)}  \|S_3(f_4)(x)\|_{1/\theta}\;  \lesssim \;
\|f_3\|_{1/(1-\theta)}
\|f_4\|_{1/{\theta}}$$
for any $0\leq \theta < 1.$

Now we consider the relatively easy case $f_3 \equiv 0$ on $5I_T$. A proof of this case is essentially similar to the proof of Lemma 9.1 in \cite{MTT2} modulo the extra constraint $2^{\#}|\omega_{Q_3}| \sim |\omega_{P_1}|$ from the definition of
$a_{Q_3}^{(3),\#}$. This extra constraint does not affect the proof in this case because we know that the collections of intervals, $\{ \omega_{P_2}\}_{ \vec \P, T}$ and $\{ \omega_{P_3}\}_{ \vec \P, T}$, are lacunary. 
For fixed $\vec Q \in T$, because of the decay of $ \Phi_{P_1}$ and  $\Phi_{Q_3}$, we have
$$|a_{Q_3}^{(3),\#}| \lesssim 
|I_{\vec P}|^{-1/2} 
\sum_{\vec P \in \vec\P:\omega_{Q_3}\subseteq \omega_{P_1} \atop 2^{\#}|\omega_{Q_3}| \sim |\omega_{P_1}|}
|\langle f_3 , \Phi_{P_2}\rangle |
|   \langle f_4 , \Phi_{P_3}\rangle|
 \int \frac {\tilde\chi_{I_{\vec P}}^{100M}} {|I_{\vec P}|^{1/2}}  \frac{ \tilde\chi_{I_{\vec Q}}^{100M} }{|I_{\vec Q}|^{1/2} }dx$$
 $$\;\;\;\;\;\;=
|I_{\vec Q}|^{-1/2} 
\int
\sum_{\vec P \in \vec\P:\omega_{Q_3}\subseteq \omega_{P_1} \atop 2^{\#}|\omega_{Q_3}| \sim |\omega_{P_1}|}
\left(
|\langle f_3 , \Phi_{P_2}\rangle |
|   \langle f_4 , \Phi_{P_3}\rangle|
 \frac {\tilde\chi_{I_{\vec P}}^{100M}} {|I_{\vec P}|}  
\right)
 \tilde\chi_{I_{\vec Q}}^{100M} dx.$$
 Furthermore, by applying Cauchy-Schwarz inequality, we thus have
 $$|a_{Q_3}^{(3),\#}| \lesssim 
|I_{\vec Q}|^{-1/2} 
 \int |S_2f_3||S_3f_4|\tilde\chi_{I_{\vec Q}}^{100M}dx,$$
 where the square function $S_j$, for $j = 2, 3$, is the vector-valued quantity
 $$S_jf := \left(\langle f, \Phi_{P_j}\rangle 
 \frac { \tilde\chi_{I_{\vec P}}^{50M}} {|I_{\vec P}|^{1/2}}  \right)_{\vec P \in \vec\P:\omega_{Q_3}\subseteq \omega_{P_1},\;  2^{\#}|\omega_{Q_3}| \sim |\omega_{P_1}|}.$$ 
 We claim the weighted square-function estimate
 $$\| S_jf  \|_{L^P( \tilde\chi_{I_{\vec Q}}^{100M} dx)}
 \lesssim \|f\|_{L^P( \tilde\chi_{I_{\vec Q}}^{M} dx)},$$
 for all $1 < p < \infty$ and $j = 2,3$, by H\"older, which follows because the collection $\{\omega_{P_j}\}_{ \vec \P, T}$ of intervals, for each $j=2,3$, is lacunary around some frequency $\xi_o$. So $S_j$ is a modulated Calder\'{o}n-Zygmund operator whose kernel is
 $$K_j(x, y) = \bar\Phi_{P_j}(y) \frac { \tilde\chi_{I_{\vec P}}^{50M}(x)} {|I_{\vec P}|^{1/2}},$$ 
which decays like $O(|I_{\vec Q}|^{-1}(|x - y|/|I_{\vec Q}|)^{-50M})$ for all
$|x -y| \gg |I_{\vec Q}|$.

Therefore, we see that 
$$|a_{Q_3}^{(3),\#}| \lesssim
|I_{\vec Q}|^{-1/2}
  \|f_3\|_{L^{\frac 1 {1-\theta}}( \tilde\chi_{I_{\vec Q}}^{M} dx)}
    \|f_4\|_{L^{\frac 1 {\theta}}( \tilde\chi_{I_{\vec Q}}^{M} dx)}$$
    $$\;\;\;
    \lesssim
|I_{\vec Q}|^{-1/2} 
 ( \int_{E_3}\tilde\chi_{I_{\vec Q}}^M)^{1-\theta}
  ( \int_{E_4}\tilde\chi_{I_{\vec Q}}^M)^{\theta}$$
$$ \;\;\;\;\;\;\;\;\;\;\;\;\;\;\; \;\;\;\;\;\;\;\;
\lesssim
  |I_{\vec Q}|^{-1/2} \left(  \frac { |I_{\vec Q}|} { |I_{T}|}
   \right)^{M(1-\theta)}
  ( \int_{E_3}\tilde\chi_{I_{T}}^M)^{1-\theta}
  ( \int_{E_4}\tilde\chi_{I_T}^M)^{\theta}$$
for all $\vec Q \in T$; the claim (\ref{20}) then follows by square-summing in $\vec Q$.
This proves the estimate (\ref{20}) when $f_3 \equiv 0$ on $5I_T$. A similar argument gives the estimate (\ref{20}) when $f_4 \equiv 0$ on $5I_T.$ 
\end{proof}

We will now prove the claim in the proof of Lemma 8.2.
\begin{lemma}
Let $T \subset \vec \Q$ be sparse i-tree for some $i=1,2$. And let
$$\vec \P' (T) = \{\vec P'=\tilde{P_1} \times P_2 \times {P_3} : \vec P =P_1 \times P_2 \times P_3  \in \vec \P \;{\rm and} \;  \tilde{P_1}=I_{\vec P,{2^{\#}}} \times  \omega_{Q_3} $$
 $$\;{\rm for}\; {\rm some}\; \vec Q \in T\; {\rm with} \;
 \omega_{Q_3}\subseteq \omega_{P_1} \;{\rm and }\; \;2^{\#}|\omega_{Q_3}| \sim |\omega_{P_1}|\},$$ 
 where $I_{\vec P,{2^{\#}}}$ is the unique dyadic interval of length
  $2^{\#}|I_{\vec P}|$ containing $I_{\vec P}$. 
 Then by assuming $\vec Q \in T$ and $\vec P \in \vec \P$, 
 we have the following statement:
 $$\omega_{Q_3}\subseteq \omega_{P_1}\; {\rm and}\;  2^{\#}|\omega_{Q_3}| \sim |\omega_{P_1}|
 \;\;\;{\rm if \;and\; only\; if }\;\;\;
 \vec P' \in \vec \P' \;{\rm and }\; \langle \tilde \Phi_{\tilde P_1},  \Phi_{Q_3}\rangle \neq 0.$$
 \end{lemma}
We remark that the collection $\vec \P'$ depends only on $\vec \P$ and $T \subset \vec \Q$, and not on each $\vec Q \in T$.
\vskip .5cm

\begin{proof}
Let $\vec Q \in T$ and $\vec P \in \vec \P$. And let $\vec P' \in \vec \P'$ such that $\langle \tilde \Phi_{\tilde P_1},  \Phi_{Q_3}\rangle \neq 0$. Then by definition of $\vec \P'$, we know that there exists a $\vec{{Q^*}} \in T$ such that 
$\omega_{Q^*_3}\subseteq \omega_{P_1}, \; 2^{\#}|\omega_{Q^*_3}| \sim |\omega_{P_1}|$
 and $\tilde{P_1}=I_{\vec P,{2^{\#}}} \times  \omega_{Q^*_3}$. 
Since $\langle \tilde \Phi_{\tilde P_1},  \Phi_{Q_3}\rangle \neq 0$, we have that $\omega_{Q^*_3} \cap \omega_{Q_3} \neq \emptyset$. Thus we obtain that $\omega_{Q^*_3}= \omega_{Q_3}$ because both $\vec Q$ and $\vec{{Q^*}}$ are in $T$ which is sparse $i$-tree for some $i=1,2$. Therefore, we have that $\omega_{Q_3}\subseteq \omega_{P_1}$ and $2^{\#}|\omega_{Q_3}| \sim |\omega_{P_1}|.$
This proves the ``if " part. Since the opposite implication is obvious, we have completed the proof.
\end{proof}
We will give the proof of the energy estimates.
\begin{lemma}
For $ j = 3,4$, let $E_j$ be sets of finite measure in $\R$ and $f_j$ be functions in $X(E_j)$.
Then, we have
$${\energy_3}((  a^{(3), \#}_{Q_3} )_{ \vec Q\in\vec \Q} ) \lesssim 
2^{\#/2}
\left( |E_4|^{1/2} \sup_{\vec P \in \vec \P} \frac{\int_{E_3}\tilde{ \chi}_{I_{\vec P}}^M} {|I_{\vec P}|} \right)^{1-\theta}
\left( |E_3|^{1/2} \sup_{\vec P \in \vec \P} \frac{\int_{E_4}\tilde{ \chi}_{I_{\vec P}}^M} {|I_{\vec P}|} \right)^{\theta}
$$
for any $0 < \theta < 1$ and $M > 0$, with the implicit constant depending on $\theta, M$. In particular, we have
$
{\energy_3}((  a^{(3), \#}_{Q_3} )_{ \vec Q\in\vec \Q} ) \lesssim 
2^{\#/2}
 |E_4|^{(1-\theta)/2} |E_3|^{\theta/2}
$
for any $0 < \theta < 1$, with the implicit constant depending on $\theta$.
\end{lemma}
\begin{proof}
By Lemma 7.5, it suffices to show that 
\begin{equation}\label{24}
| \sum_{T \in \T} \sum_{\vec Q \in T} a^{(3), \#}_{Q_3} \bar{c}_{Q_3} |
 \lesssim
 2^{\#/2}
  \left( |E_4|^{1/2} \sup_{\vec P \in \vec \P} \frac{\int_{E_3}\tilde{ \chi}_{I_{\vec P}}^M} {|I_{\vec P}|} \right)^{1-\theta}
\left( |E_3|^{1/2} \sup_{\vec P \in \vec \P} \frac{\int_{E_4}\tilde{ \chi}_{I_{\vec P}}^M} {|I_{\vec P}|} \right)^{\theta}
\end{equation}
for all collections $\T$ of strongly 3-disjoint trees and all coefficients $c_{Q_3}$ such that
$$\sum_{\vec Q \in \tilde T} |c_{Q_3}|^2 \lesssim \frac{|I_{\tilde T}|} {\sum_{T \in \T} |I_T|}$$
for all $T \in \T$ and all sub-trees $\tilde T \subseteq T$.
Now fix $\T$ and $c_{Q_3}$. Then,
\begin{align}
| \sum_{T \in \T} \sum_{\vec Q \in T} a^{(3), \#}_{Q_3} \bar{c}_{Q_3} | &=
|\sum_T \sum_{\vec Q}
\sum_{\vec{P}\in\vec{\P}; \; \omega_{Q_3} \subset \omega_{P_1}   \atop   2^\#| \omega_{Q_3}| \sim|\omega_{P_1}| }
\frac{1}{|I_{\vec{P}}|^{1/2}}
\langle f_3, \Phi_{P_2} \rangle
\langle f_4, \Phi_{P_3} \rangle
\langle \Phi_{P_1},\; {\Phi}_{Q_3} \rangle
\bar{c}_{Q_3} | \nonumber\\
&= |\sum_{\vec{P}\in\vec{\P}} 
\frac{1}{|I_{\vec{P}}|^{1/2}}
b_{P_1}^{(1)}
\langle f_3, \Phi_{P_2} \rangle
\langle f_4, \Phi_{P_3} \rangle|,\label{25}
\end{align}
where $$b_{P_1}^{(1)}:=
 \sum_T \sum_{\vec Q \in T; \; \omega_{Q_3} \subset \omega_{P_1}   \atop   2^\#| \omega_{Q_3}| \sim|\omega_{P_1}| }
\langle \Phi_{P_1},\; c_{Q_3}{\Phi}_{Q_3} \rangle.
$$
Then, by Proposition 7.6, (\ref{25}) can be majorized by
\begin{align}
\size_1&((b_{P_1}^{(1)})_{ \vec P\in \vec \P} )^{\theta_1}
{\energy}_1((b_{P_1}^{(1)} )_{ \vec P\in \vec \P} )^{1-\theta_1}
\cdot \size_2(\langle f_3, \Phi_{P_2} \rangle_{ \vec P\in \vec \P} )^{\theta_2}
\cdot\nonumber\\
&{\energy}_2(\langle f_3, \Phi_{P_2} \rangle_{ \vec P\in \vec \P} )^{1-\theta_2}
\size_3(\langle f_4, \Phi_{P_3} \rangle_{ \vec P\in \vec \P} )^{\theta_3}
{\energy}_3(\langle f_4, \Phi_{P_3} \rangle_{ \vec P\in \vec \P} )^{1-\theta_3} \label{26}
\end{align}
for any $0 \leq \theta_1, \theta_2, \theta_3< 1$ with $\theta_1+\theta_2+\theta_3=1$.
Furthermore, by Lemma 7.7 and Lemma 7.8, we obtain that (\ref{26}) can be majorized by
\begin{align}
\size_1((b_{P_1}^{(1)})_{ \vec P\in \vec \P} )^{\theta_1}
{\energy}_1((b_{P_1}^{(1)} )_{ \vec P\in \vec \P} )^{1-\theta_1}
  \left( \sup_{\vec P \in \vec \P} \frac{\int_{E_3}\tilde{ \chi}_{I_{\vec P}}^M} {|I_{\vec P}|} \right)^{\theta_2}
\left( \sup_{\vec P \in \vec \P} \frac{\int_{E_4}\tilde{ \chi}_{I_{\vec P}}^M} {|I_{\vec P}|} \right)^{\theta_3}\|f_3\|_2^{1-\theta_2} \|f_4\|_2^{1-\theta_3}.\label{27}
\end{align}
By setting $\theta_1=0$ and $\theta_3:=\theta$, we have that (\ref{27}) is equal to
$$
{\energy}_1((b_{P_1}^{(1)} )_{ \vec P\in \vec \P} )
 \left( \sup_{\vec P \in \vec \P} \frac{\int_{E_3}\tilde{ \chi}_{I_{\vec P}}^M} {|I_{\vec P}|} \right)^{1-\theta}
\left( \sup_{\vec P \in \vec \P} \frac{\int_{E_4}\tilde{ \chi}_{I_{\vec P}}^M} {|I_{\vec P}|} \right)^{\theta}
\|f_3\|_2^{\theta} \|f_4\|_2^{1-\theta} $$
$$\lesssim
{\energy}_1((b_{P_1}^{(1)} )_{ \vec P\in \vec \P} )
 \left( |E_4|^{1/2} \sup_{\vec P \in \vec \P} \frac{\int_{E_3}\tilde{ \chi}_{I_{\vec P}}^M} {|I_{\vec P}|} \right)^{1-\theta}
\left( |E_3|^{1/2} \sup_{\vec P \in \vec \P} \frac{\int_{E_4}\tilde{ \chi}_{I_{\vec P}}^M} {|I_{\vec P}|} \right)^{\theta}.
$$
So, to establish (\ref{24}), we are left to show that 
$${\energy}_1((b_{P_1}^{(1)} )_{ \vec P\in \vec \P} ) \lesssim 
 2^{\#/2}.
$$
By Lemma 7.5 again, we obtain that there exist a collection $\T'$ of strongly 1-disjoint trees and a complex coefficient $d_{P_1}$, for all $\vec P \in \bigcup_{T' \in \T'} T'$, such that 
$${\energy}_1((b_{P_1}^{(1)} )_{ \vec P\in \vec \P} ) \sim 
| \sum_{T' \in \T'} \sum_{\vec P \in T'} b_{P_1} \bar{d_{P_1}} | $$
and such that 
$$\sum_{\vec P \in \tilde T''} |d_{P_1}|^2 \lesssim \frac{|I_{ T''}|} {\sum_{T' \in \T'} |I_{T'}|}$$
for all $T' \in \T' $ and all sub-trees $T'' \subseteq T'$  of $T'$.

Then, we reduce to showing that 
\begin{equation}\label{28}
| \sum_{T' \in \T'} \sum_{\vec P \in T'} b_{P_1} \bar{d_{P_1}} | =
| \sum_{T' \in \T'} \sum_{\vec P \in T'}  \sum_{T \in \T} \sum_{\vec Q \in T; \; \omega_{Q_3} \subset \omega_{P_1}   \atop   2^\#| \omega_{Q_3}| \sim|\omega_{P_1}| }
c_{Q_3} d_{P_1} 
\langle  \Phi_{P_1} {\Phi}_{Q_3} \rangle|
\lesssim
2^{{\#}/2}.
\end{equation}
From the decay of the ${\Phi}_{Q_3}$ we have
$$|\langle  \Phi_{P_1} {\Phi}_{Q_3} \rangle|
\lesssim
 \frac { |I_{\vec P}|^{1/2} }{ |I_{\vec Q}|^{1/2} } \left( 1+\frac {\dist(I_{\vec Q}, I_{\vec P})} {|I_{\vec Q}|} \right)^{-100}$$
 $$\lesssim 2^{{-\#}/2} 2^{-100k}, $$
 if we assume $\dist(I_{\vec Q}, I_{\vec P}) \sim 2^k |I_{\vec Q}|$ for some $k \in \Z$.
 
 Now, choose any $T' \in \T'$ and pick any $\vec P \in T'$. Then, we construct sets $\A_{\vec P, \T}^k$ defined by
 $$\A_{\vec P, \T}^k := \{ \vec Q \in \bigcup_{T \in \T} T : \omega_{Q_3} \subset \omega_{P_1} ,\:   2^\#| \omega_{Q_3}| \sim|\omega_{P_1}| \; {\rm{and}} \; \dist(I_{\vec Q}, I_{\vec P}) \sim 2^k |I_{\vec Q}| \},$$
 for $k \geq 1$.
 Then, we can see that there is a bounded number $C_{\#, k}$ for $\vec Q$ satisfying all assumptions in the set for each $\vec P \in T'$. More specifically, we establish that $C_{\#, k} \lesssim 2^k \cdot 2^{\#}$.
 
 Hence, using the Cauchy-Schwarz inequality, we can estimate the corresponding piece of (\ref{28}) by
\begin{align}
 &| \sum_{T' \in \T'} \sum_{\vec P \in T'} \sum_{ \vec Q \in \A_{\vec P, \T}^k}
c_{Q_3} d_{P_1} 
\langle  \Phi_{P_1} {\Phi}_{Q_3} \rangle|\nonumber\\
\lesssim &\;\;
C_{\#, k} 2^{{-\#}/2} 2^{-100k}  (\sum_{T' \in \T'} \sum_{\vec P \in T'}  |d_{P_1}|^2)^{1/2} (\sum_{T \in \T} \sum_{\vec Q \in T}  |c_{Q_3}|^2)^{1/2}\nonumber\\
\lesssim & \;\;2^{{\#}/2} 2^{-99k}. 
\end{align}
Summing over $k \geq 1$, we have
$$| \sum_{T' \in \T'} \sum_{\vec P \in T'}  \sum_{T \in \T} \sum_{\vec Q \in T; \; \omega_{Q_3} \subset \omega_{P_1}  \;    2^\#| \omega_{Q_3}| \sim|\omega_{P_1}| \atop \dist(I_{\vec P}, I_{\vec Q}) \gtrsim |I_{\vec Q}|}
c_{Q_3} d_{P_1} 
\langle  \Phi_{P_1} {\Phi}_{Q_3} \rangle|
\lesssim
2^{{\#}/2}.$$
If $\dist(I_{\vec P}, I_{\vec Q}) \lesssim |I_{\vec Q}|$, then we have $I_{\vec P} \subseteq 3 I_{\vec Q}$ and thus
$$ |\langle  \Phi_{P_1} {\Phi}_{Q_3} \rangle| \lesssim 2^{{-\#}/2}.$$
Hence we can estimate the corresponding piece of (\ref{28}) by
$$| \sum_{T' \in \T'} \sum_{\vec P \in T'}  \sum_{T \in \T} \sum_{\vec Q \in T; \; \omega_{Q_3} \subset \omega_{P_1}  \;    2^\#| \omega_{Q_3}| \sim|\omega_{P_1}| \atop \dist(I_{\vec P}, I_{\vec Q}) \leq |I_{\vec Q}|}
c_{Q_3} d_{P_1} 
\langle  \Phi_{P_1} {\Phi}_{Q_3} \rangle|
\lesssim
3 \cdot 2^{\#} 2^{-{\#}/2} \lesssim 2^{{\#}/2}.$$
This completes the proof of the energy estimates.
\end{proof}

\vskip 1cm

\section{Proof of Theorem 3.5 }

We now use the estimates from the previous two sections to prove the main theorem. We have few cases to consider depending on the bad index of each vertex $A_i$, $\tilde{A_i}$ for $1<i<12$. Here, we only consider the case of vertex $A_i$ because the other case can be done analogously.

\underline{\emph{\textbf{(Case 1: vertices with a bad index 1)}}}:  the case of bad index 2 can be dealt with similarly.

Fix an admissible tuple $\alpha= (\alpha_1, \alpha_2, \alpha_3, \alpha_4)$ near $A_i$ for some $9 \leq i \leq 12$.  Fix also the finite collections $\vec \P$, $\vec \Q$ of tri-tiles and an arbitrary tuple $(E_1,E_2,E_3,E_4)$ of subsets of $\R$ with
finite measure. We need  to find a major subset $E'_1$ of $E_1$ such that
$$|\Lambda^{\#}_{\vec \P, \vec \Q} (f_1, f_2, f_3, f_4)| \lesssim 2^{\#/2}|E_1|^{\alpha_1} |E_2|^{\alpha_2}|E_3|^{\alpha_3}|E_4|^{\alpha_4}$$
for all tuple $(f_1, f_2, f_3, f_4)$ of functions with $f_1 \in X(E'_1)$ and $f_i \in X(E_i)$, $i=2,3,4$.

Define the exceptional set $\Omega_{C_1}$ by
$$ \Omega_{C_1}:= \bigcup_{j=1}^4 \{M_{\chi_{E_j} }> C_1|E_j|/|E_1|\}$$
for a large constant $C_1$, where $M$ is the dyadic Hardy-Littlewood maximal function. By the classical Hardy-Littlewood inequality, we have $|\Omega_{C_1}| < |E_1|/2$ if ${C_1}$ is sufficiently large. Thus, if we set $E'_1 := E_1\setminus \Omega_{C_1}$ for a sufficiently large constant, then $E'_1$ is a major subset of $E_1$. Now, we shall show that, for fixed $f_1 \in X(E'_1)$ and $f_i \in X(E_i)$, 
 for $i = 2, 3, 4$, we establish the estimate
 \begin{equation}\label{29}
 | \sum_{\vec Q \in \vec\Q} \frac {1}{|I_{\vec Q}|^{1/2}}
a^{(1)}_{Q_1}a^{(2)}_{Q_2}a^{(3), \#}_{Q_3}| 
\lesssim 2^{\#/2} |E_1|^{\alpha_1} |E_2|^{\alpha_2}|E_3|^{\alpha_3}|E_4|^{\alpha_4},
\end{equation}
where the $a^{(1)}_{Q_1}$, $a^{(2)}_{Q_2}$ and $a^{(3),\#}_{Q_3}$ are defined by (\ref{tiles}).

We first consider a decomposition of $ \vec\Q$ as  $\bigcup_{k\geq0}  \vec\Q_k$, where, for each $k \geq 0$, the finite collection $\vec\Q_k$ of tri-tiles is given by $$\vec\Q_k:=\{\vec Q \in \vec\Q : 1 +\frac{\dist(I_{\vec Q} , \R \setminus \Omega_{C_1})} {| I_{\vec Q}|} \sim 2^k \}.$$

We will establish such corresponding estimate of $\vec\Q_k$ as (\ref{29}) with an additional factor of $2^{-k}$ on the right-hand side, which enables us to complete our proof after summing over $k$.

Fix $k\geq0$. For any $\vec Q \in \vec\Q_k$, we have
$$\frac{\int_{E'_1} \tilde \chi_{I_{\vec Q}}^{M_1} } {|I_{\vec Q}|} \lesssim 2^{-({M_1}+C_1)k},$$
and
$$\frac{\int_{E_j} \tilde \chi_{I_{\vec Q}}^{M_1} } {|I_{\vec Q}|} \lesssim 2^k \frac {|E_j|}{|E_1|}$$
for $j = 2, 3, 4$ and for a sufficiently big number ${M_1}$.

By Lemma 7.7 and Lemma 8.2 we thus have that
\beq
\size_1((a^{(1)}_{Q_1})_{ \vec Q\in \vec \Q_k}) &\lesssim& 2^{-({M_1}+C_1)k}\\
\size_2((a^{(2)}_{Q_2})_{ \vec Q\in \vec \Q_k}) &\lesssim& 2^k \frac {|E_2|}{|E_1|}\\
\size_3((a^{(3),\#}_{Q_3})_{ \vec Q\in \vec \Q_k}) &\lesssim& 2^k \frac {|E_3|^{1-\theta}|E_4|^{\theta}}{|E_1|}
\eeq
for some $0<\theta <1$. Similarly, by Lemma 7.8, Lemma 8.4, we have
\beq
{\energy_1}((  a^{(1)}_{Q_1} )_{ \vec Q\in\vec \Q_k} )&\lesssim& {|E_1|}^{1/2}\\
{\energy_2}((  a^{(2)}_{Q_2} )_{ \vec Q\in\vec \Q_k} )&\lesssim& {|E_2|}^{1/2}\\
{\energy_3}((  a^{(3), \#}_{Q_3} )_{ \vec Q\in\vec \Q_k} ) &\lesssim& 
2^{\#/2}{|E_3|}^{(1-\theta)/2}{|E_4|}^{\theta/2}\eeq

By Proposition 7.6, for a sufficiently large $M_1$, we establish that
$$| \sum_{\vec Q \in \vec\Q_k} \frac {1}{|I_{\vec Q}|^{1/2}}
a^{(1)}_{Q_1}a^{(2)}_{Q_2}a^{(3), \#}_{Q_3}| 
\lesssim  2^{\#/2}2^{-k} \frac 
{|E_1|^{(1+\theta_1)/2}{|E_2|}^{(1+\theta_2)/2}({|E_3|}^{(1-\theta)/2}{|E_4|}^{\theta/2})^{(1+\theta_3)/2} }
{|E_1|}
$$
for $0< \theta_1, \theta_2, \theta_3<1$ with $ \theta_1+ \theta_2+ \theta_3=1$. By choosing $ \theta_1:=2\alpha_1+1,  \theta_2:=2\alpha_2-1,  \theta_3:=2(\alpha_3+\alpha_4)-1,$ and $\theta:=\frac{\alpha_4}{(\alpha_3+\alpha_4)}$, we obtain the corresponding estimates for $\vec\Q_k$ in (\ref{29}).
\vskip .5cm
\underline{\emph{\textbf{(Case 2: vertices with a bad index 4)}}}:  the bad index 3 case can be dealt with similarly.

Fix an admissible tuple $\alpha= (\alpha_1, \alpha_2, \alpha_3, \alpha_4)$ near $A_i$ for some $i=1,2$.  Fix also the finite collections $\vec \P$, $\vec \Q$ of tri-tiles and an arbitrary tuple $(E_1,E_2,E_3,E_4)$ of subsets of $\R$ with
finite measure. 
As before, define the exceptional set $\Omega_{C_2}$ by
$$ \Omega_{C_2}:= \bigcup_{j=1}^4 \{M_{\chi_{E_j} }> C_2|E_j|/|E_4|\}$$
for a constant $C_2$, and set $E'_4 := E_4\setminus \Omega_{C_2}$. Then $E_4'$ is a major subset of $E_4$, if ${C_2}$ is sufficiently large. Now, we will show that for fixed $f_4 \in X(E'_4)$ and $f_i \in X(E_i)$, 
 for $i = 1, 2, 3$, we establish the estimate
 \begin{equation}\label{30}
 | \sum_{\vec Q \in \vec\Q} \frac {1}{|I_{\vec Q}|^{1/2}}
a^{(1)}_{Q_1}a^{(2)}_{Q_2}a^{(3),\#}_{Q_3}| 
\lesssim 2^{\#/2} |E_1|^{\alpha_1} |E_2|^{\alpha_2}|E_3|^{\alpha_3}|E_4|^{\alpha_4},
\end{equation}
where the $a^{(1)}_{Q_1}$, $a^{(2)}_{Q_2}$ and $a^{(3),\#}_{Q_3}$ are given by (\ref{tiles}).

We now obtain a decomposition of $ \vec\Q$ as  $\bigcup_{k\geq0}  \vec\Q_k$, where the finite collection $\vec\Q_k$ of tri-tiles is given by $$\vec\Q_k:=\{\vec Q \in \vec\Q : 1 +\frac{\dist(I_{\vec Q} , \R \setminus \Omega_{C_2})} {| I_{\vec Q}|} \sim 2^k \},$$
for each $k \geq 0$, and similarly obtain a decomposition of $ \vec\P$ as  $\bigcup_{k'\geq0}  \vec\P_{k'}$, where the finite collection $\vec\P_{k'}$ of tri-tiles is given by $$\vec\P_{k'}:=\{\vec P \in \vec\P : 1 +\frac{\dist(I_{\vec P} , \R \setminus \Omega_{C_2})} {| I_{\vec P}|} \sim 2^{k'} \},$$
for each $k' \geq 0$.
We will establish such corresponding estimate of $\vec\Q_k$ and $\vec\P_{k'}$ as (\ref{30}) with an additional factor of $2^{-k-k'}$ on the right-hand side, which enables us to complete our proof after summing over $k$ and $k'$.
 
 Fix $k\geq0$. For any $\vec Q \in \vec\Q_k$, we have
 $$\frac{\int_{E'_4} \tilde \chi_{I_{\vec Q}}^{M_2} } {|I_{\vec Q}|} \lesssim 2^{-({M_2}+C_2)k}$$
$$\frac{\int_{E_j} \tilde \chi_{I_{\vec Q}}^{M_2} } {|I_{\vec Q}|} \lesssim 2^k \frac {|E_j|}{|E_4|}$$
for $j =1, 2, 3, $ and for a sufficiently big number ${M_2}$.

By Lemma 7.7 and Lemma 8.2 we have
$$\size_1((a^{(1)}_{Q_1})_{ \vec Q\in \vec \Q_k}) \lesssim 2^k \frac {|E_1|}{|E_4|}$$
$$\size_2((a^{(2)}_{Q_2})_{ \vec Q\in \vec \Q_k}) \lesssim 2^k \frac {|E_2|}{|E_4|}$$
$$\size_3((a^{(3),\#}_{Q_3})_{ \vec Q\in \vec \Q_k}) \lesssim 2^{-(M_2+C_2)k}$$
for some $0<\theta <1$. Similarly, by Lemma 7.8, Lemma 8.4, we have
$${\energy_1}((  a^{(1)}_{Q_1} )_{ \vec Q\in\vec \Q_k} )\lesssim {|E_1|}^{1/2}$$ 
$${\energy_2}((  a^{(2)}_{Q_2} )_{ \vec Q\in\vec \Q_k} )\lesssim {|E_2|}^{1/2}.$$ 

Now, fix $k'\geq0$. For any $\vec P \in \vec\P_{k'}$, we obtain that
$$\frac{\int_{E_3} \tilde \chi_{I_{\vec P}}^{M_2} } {|I_{\vec P}|} \lesssim 2^{k'} \frac {|E_3|}{|E_4|}$$
$$\frac{\int_{E'_4} \tilde \chi_{I_{\vec P}}^{M_2} } {|I_{\vec P}|} \lesssim 2^{-({M_2}+C_2){k'}}.$$
By Lemma 8.4 we thus have
$${\energy_3}((  a^{(3), \#}_{Q_3} )_{ \vec Q\in\vec \Q_k}, \vec P \in \vec\P_{k'} ) \lesssim 
2^{\#/2} 2^{-(M_2+C_2)k'}{|E_3|}^{(2-\theta)/2}{|E_4|}^{{(\theta-1)}/2}$$ 
for some $0<\theta<1$ to be chosen later.

By Proposition 7.6, for a sufficiently large $M_2$, we establish that
$$| \sum_{\vec Q \in \vec\Q_k, \vec P \in \vec\P_{k'}} \frac {1}{|I_{\vec Q}|^{1/2}}
a^{(1)}_{Q_1}a^{(2)}_{Q_2}a^{(3), \#}_{Q_3}| 
\lesssim
2^{\#/2}2^{-k-k'} \frac 
{|E_1|^{(1+\theta_1)/2}{|E_2|}^{(1+\theta_2)/2}}{|E_4|^{1-\theta_3}}
({|E_3|}^{(2-\theta)/2}{|E_4|}^{{(\theta-1)}/2})^{1-\theta_3}
$$
for $0< \theta_1, \theta_2, \theta_3<1$ with $ \theta_1+ \theta_2+ \theta_3=1$. By setting $ \theta_1:=2\alpha_1-1,  \theta_2:=2\alpha_2-1,  \theta_3:=2(\alpha_3+\alpha_4)-1,$ and $\theta:=\frac{(3\alpha_3+2\alpha_4)}{\alpha_3+\alpha_4}$, we obtain the corresponding estimates for $\vec\Q_k$ and $\vec\P_{k'}$ in (\ref{30}).









\end{document}